\documentclass{amsart}
\usepackage{vmargin,amsmath,amsfonts,amssymb,psfrag,amsthm,xcolor,amscd,latexsym,xspace,enumerate}
\usepackage[pagebackref, colorlinks=true,linkcolor=blue,citecolor=red]{hyperref}
\usepackage{lipsum}
\usepackage{adjustbox}
\usepackage{pdfpages}
\usepackage{tikz-cd}
\usetikzlibrary{arrows}
\usetikzlibrary{positioning}
\usetikzlibrary{graphs}
\usetikzlibrary{intersections}
\usepackage{pgfplots}
\usetikzlibrary{calc,3d,shapes, pgfplots.external, intersections}

\usepackage[mathscr]{eucal}

\tikzset{negated/.style={
        decoration={markings,
            mark= at position 0.5 with {
                \node[transform shape] (tempnode) {$\backslash$};
            }
        },
        postaction={decorate}
    }
}
\tikzset{
    vertex/.style={draw,circle,thick, inner sep=1.5pt,minimum size=6pt,
    },
    edge/.style={thick},
    dedge/.style = {->,> = latex',thick}
}
\definecolor{asparagus}{rgb}{0.53, 0.66, 0.42}
\usetikzlibrary{decorations.markings}
\usetikzlibrary{decorations.text}
\pgfplotsset{compat=1.18}

\newtheorem{theo}{Theorem}[section]
\newtheorem{thm}[theo]{Theorem}
\newtheorem{lem}[theo]{Lemma}
\newtheorem{prop}[theo]{Proposition}
\newtheorem{cor}[theo]{Corollary}
\newtheorem{defn}[theo]{Definition}

\theoremstyle{definition}
\newtheorem{rem}[theo]{Remark}

\newtheorem{example}[theo]{Example}
\newtheorem{problem}[theo]{Problem}

\numberwithin{equation}{section}

\begin{document}

\title[On the subdirect product of graph bundles]{On the subdirect product of  graph bundles}

\author[Y. Bavuma]{Yanga Bavuma}
\address{Yanga Bavuma\endgraf
Department of Mathematics and Applied Mathematics\endgraf
University of Cape Town\endgraf
Private Bag X1, Rondebosch, 7701\endgraf
Cape Town, South Africa\endgraf}

\author[F.G. Russo]{Francesco G. Russo}
\address{Francesco G. Russo\endgraf
School of Science and Technology\endgraf
Universit\'a degli Studi di Camerino\endgraf
via Madonna delle Carceri 9, 62032\endgraf 
Camerino, Italy\endgraf
and\endgraf
Department of Mathematics and Applied Mathematics\endgraf
University of the Western Cape\endgraf
Private Bag X17, 7535,  Bellville, South Africa\endgraf}

\author[S. Spessato]{Stefano Spessato}
\address{Stefano Spessato\endgraf
Dipartimento di Ingegneria\endgraf
Universit\`{a} degli Studi Niccol\`{o} Cusano \endgraf
Via Don Carlo Gnocchi 3, 00166 \endgraf
Rome, Italy\endgraf}

\keywords{Subdirect product; graph bundles; $K$-theory; Cayley graphs\endgraf
\textit{Mathematics Subject Classification 2020}: Primary 05C76, 05C25; Secondary 05C90}

\begin{abstract}
The subdirect product of two finite groups $A$ and $B$ is defined as a subgroup of the direct product $A \times B$, which is a well-known notion in finite group theory. While it is clear that, under appropriate choices of sets of generators $S$, $S_A$ and $S_B$, the Cayley graph $Cay(A \times B, S)$ corresponds to the Cartesian product $Cay(A, S_A) \square Cay(B, S_B)$ of two graphs, there is no analogue at the level of graph product that reflects the notion of subdirect product of groups. This is precisely the problem which we  discuss here. By using the concept of graph bundles and the corresponding pullbacks, we introduce an operation on graph bundles such that the Cayley graph of the subdirect product of two groups can be described as the total space of the product of the Cayley graphs. This allows us to define the so-called ``network $K$-theory group of a graph'', inspired by the notion of topological $K$-theory, and we are able to  investigate an interesting  functor from the category of graphs to the category of abelian groups.
\end{abstract}

\maketitle

\tableofcontents

\section{Introduction}

The origin  and the motivations of the present contribution should be found in some recent results on the study of the combinatorial structures, which  appear in  quantum information theory and  complex dynamical systems in \cite{bbr1, b, bddr}.

It has been known for a long time that the algebraic combinatorics plays a crucial role in many topics of quantum mechanics and the consideration of the finite group of order 16 of the Pauli  matrices offers a significant example of relevant connections between algebraic combinatorics, topological algebra, group theory and quantum mechanics.  Planat and others \cite{ps1, ps2} devoted large part of their investigations to the structure of Pauli groups and developed a series of generalizations of the usual Pauli group of order 16, involving many ideas from mathematical physics,  representation theory, group theory and combinatorics.

On the other hand, structural results (at a purely abstract level) provide another framework of generalization for the Pauli groups on 2-qubits and \cite{bbr1, b, bddr} show that the notion of product in graph theory and in group theory can be very useful for topics of  quantum mechanics.

It should be noted that there is a huge amount of  research in recent years, because the quantum information theory is growing in many directions, especially under the profile of the connections with topological graph theory and spectral graph theory.  Therefore there is also a growing interest in the literature for those mathematical structures which are helpful to describe qubits.

The products of groups and graphs are, in fact, the main topic of the present article. There are numerous examples of graph operations \cite{neps1, neps2, handbook, sab} as well as operations on groups \cite[Chapter A, \S 19, \S 20]{dh}. These two families of products generally have no direct relation and may exist independently on each other. However, the connections between algebra and graph theory are numerous, and several studies addressed objects and products that lie at the intersection of these two research fields, see \cite{bddr, belardo1,  godsil}.

One of the most productive directions, which was explored to establish connections between group theory and combinatorics, involves Cayley graphs. These are classical objects in geometric group theory. In fact Cayley graphs, under an appropriate choice of generators, have some geometric and spectral features which encode many algebraic properties of the underlying group. There is, therefore, a research line  on determining whether  a product of graphs corresponds to a product of groups, or not. Accordingly,  the Cayley graph of the  product of two finite groups may coincide with the graph product of the associated Cayley graphs, or not (see \cite{alon, donno1}). The present contribution fits precisely  this line of research.

We should mention that it is not immediate to  recognize the structure of subdirect product among finite groups, see \cite[Chapter A, \S 19, \S 20]{dh}. In fact  there are several finite groups that present sophisticated structures and different forms of products. For instance the aforementioned Pauli group of order 16 (which is a relatively small size for a finite group) has a very rich structure at the level of abelian subgroups and abelian quotients. Already in similar situations, it is difficult to recognize an appropriate structure of product, but  once this process of recognition of a product is clarified, we may easily classify abelian subgroups and quotients. 

In principle, one could extend the notion of subdirect product to categories, which are different from the category of finite groups (with corresponding group homomorphisms). When we do this,  a series of complications emerge naturally. First of all, we should formalize what we mean for ``subdirect product'' in a category and which morphisms of the category we shall consider. It is well known that the category of topological groups (see \cite{hofmor}) possesses homomorphisms of groups which are not necessarily continuous. Therefore one could be in situations, where certain morphisms preserve a structure on a set, but not another.  A more delicate problem is to select appropriate contexts, where products indeed exist. In fact it is also well known that there are categories for which this is not always the case, see \cite[Appendix 3]{hofmor}.  

In the present manuscript our focus, from an algebraic point of view, is on the subdirect product of finite groups. This operation, given two surjective morphisms $\phi_1:A \longrightarrow C$ and $\phi_2:B \longrightarrow C$ between finite groups, yields a new finite group $A \times_C B$ which is a subgroup of the direct product $A \times B$, which is a finite  group as well. In order to construct  $A \times_C B$, it is necessary to use the notion of  \textit{pullback} of finite groups along a morphism: in particular the subdirect product  $A\times_C B$ can be defined as pullback, involving the  surjective morphisms $\phi_1$ and $\phi_2$.

The notion of  pullback is indeed  fundamental  in category theory, and in general, it is an interesting question to ask whether it is admissible for a prescribed category, or not. For this reason, it is easy to find different definitions of pullback: from group theory to differential geometry, and even in cohomological theories \cite{stefano}.

In order to find a construction in graph theory that is an appropriate analogue of the subdirect product of two finite groups, one must first understand how to interpret the surjective morphisms which are involved in the definitions. After a careful analysis of the existing literature, we noted that ``a good candidate of factor'' for the product, which we intend to define, would have been found   in theory of the graph bundles.

Originally the graph bundles were introduced in an unpublished manuscript of Pisanski and Vrabec and deal with a topic of great interest in topological graph theory, see \cite{kwak, kwakk, Pisanski}. Clearly the graph bundles are inspired by the concept of fiber bundles which is typical of the algebraic topology; they are objects of a purely combinatorial nature and have been studied and generalized by many authors, see \cite{tensore, larrion, lessico}. More specifically, an $F$-graph bundle is a triple $(X, p, \Gamma)$, where $X$, $F$, and $\Gamma$ are graphs, and $p: X \longrightarrow \Gamma$ is a surjective graph morphism. The distinctive feature of $F$-graph bundles is that (similarly to fiber bundles in topology) the preimage of a neighborhood of a vertex $v$ in $\Gamma$ under $p$ is isomorphic to the Cartesian product of the neighborhood itself with $F$. The definition of graph bundles (and much more concerning the graphs which we use here) can be found in Section \ref{section2}. The references  \cite{ttd, kos} will be followed, in order to make clear that we model the results on some common ground between algebraic topology, differential geometry and graph theory.

The second point to address is the correct interpretation of the concept of the pullback of a graph bundle along a graph morphism. We devote  Section \ref{section3} to its discussion, beginning with the definition of a pullback graph bundle and continuing with the verification of all the properties which are expected in terms of  commutativity of diagrams. Accordingly, our first  main result is Theorem \ref{thm0}. The section ends with  our second main result: the description of the adjacency matrix of a pullback of  graph bundlea in terms of the original graph bundle, namely Theorem \ref{1stmaintheorem}. The corresponding adjacency matrix  and   morphisms are described by our third main result, which is given by Theorem \ref{1stmaintheorem}.

In Section \ref{section4}, the subdirect product of graph bundles is introduced (and denoted by $\boxplus$), and relevant constructions  are illustrated, such as its good behavior with respect to pullbacks. This deals with our fourth main result, which is Theorem \ref{sumstab}. Additionally, a formula of the adjacency matrix is provided, presenting Theorem \ref{4thmaintheorem} which is our fifth main result.

After proving the monoidal structure of graph bundles defined on the same base, the $K$-theory of a network $\Gamma$ is introduced. The term \textit{network} is used exclusively in this context, in order to avoid confusion with  the $K$-theory of a graph when we look a it  as a topological space. In fact the $K$-theory involves some algebraic tools with different definitions and meanings depending on the contexts: it has brought enormous benefits to many branches of mathematics, from the algebraic topology, to the index theory and  to several areas of mathematical physics. The section ends with Corollary \ref{k.t.func}, where it is shown that the network $K$-theory provides a contravariant functor between the category of graphs and the category of abelian groups.
 
Finally, in Section \ref{section5}, after  we recall the notions of  Cayley graph of a finite group and  of  subdirect product of two finite groups, we explore new connections of these notions with that of  subdirect product of  two graphs. First, it is shown in Theorem \ref{6thmaintheorem} (our sixth main result) that for a given  group surjection $\phi: A \longrightarrow B$ and  generating sets $S_0$ for $\ker(\phi)$ and $S_1$ for $B$, there exists a generating set $S_\phi$ such that $(Cay(A, S_\phi), \phi, Cay(B, S_1))$ is a graph bundle with fiber $Cay(\ker(\phi), S_0)$. 

Then we prove in Theorem \ref{7thmaintheorem}, which is our seventh  main result, that given two group morphisms $\phi_1: A \longrightarrow C$ and $\phi_2: B \longrightarrow C$, the total space of the subdirect product of the graph bundles associated with these morphisms is the Cayley graph of the subdirect product of $A$ and $B$. In other words, with an appropriate choice of generating sets,
\begin{equation} Cay(A \times_C B, S_{\varphi}) = Cay(A, S_{\phi_1}) \boxplus Cay(B, S_{\phi_2}). \end{equation}
It should be mentioned that the logic of the investigation of \cite[Theorems 3.6]{bddr} follows the philosophy of invariance of the operator of Cayley graph via appropriate products, and we propose a similar approach in Theorems \ref{6thmaintheorem} and \ref{7thmaintheorem} in terms of graph bundles, namely our last two main results.

Open problems are illustrated along the manuscript, where we found that the principle of analogies with similar situations in algebraic topology and differential geometry clashes with concrete examples in graph theory: after all the discrete structures require naturally some types of morphisms which  cannot preserve the continuity.

\section{Preliminaries on graphs}\label{section2}
According to the usual definition of graph in \cite{spectrabook, crs}, we look at a \textit{graph} $\Gamma$ just as a couple  $\Gamma = (V_\Gamma, E_\Gamma)$ of two sets, where the elements of $V_\Gamma$ are called \textit{vertices} of  $\Gamma$  and those of   \begin{equation} E_\Gamma=\{ \ \{x,y\}  \mid x,y \in V_\Gamma \ \mbox{with}  \ x \neq y \ \}\end{equation}   are called   \textit{edges} of $\Gamma$. If   $\{x,y\} \in E_\Gamma$, then $x$ and $y$ are said to be \textit{adjacent}  (writing briefly $x \sim y$).  By default, we have the concept of  \textit{adjacency matrix} of a graph $\Gamma$; this is the real matrix $A_\Gamma$ indexed by the vertices of $\Gamma$ such that the entry $a_{vw} = 1$ if $v \sim w$ or $0$ otherwise. See \cite[Chapter 1]{crs}, or \cite[Chapter 1]{spectrabook}, for details on this well known notion.

The adjacency matrix $A_\Gamma$  of a graph $\Gamma$  depends  on the choice of an order on  $V_\Gamma$, but its spectrum 
\begin{equation} \mathrm{spec}(A_\Gamma)=\{\lambda \mid \lambda \ \mbox{is eigenvalue of } \ A_\Gamma\}  \end{equation}
does not depend on the presence of an order on $V_\Gamma$. This makes the spectrum of $A_\Gamma$ a significant tool, in order to study the properties of $\Gamma$. Note that the spectrum of $A_\Gamma$ is called the \textit{spectrum of } $\Gamma$  without ambiguity, since it is invariant under graph isomorphisms, and is usually denoted by $\mathrm{spec}(\Gamma)$. In other words, 
\begin{equation} \mathrm{spec}(A_\Gamma)=\mathrm{spec}(\Gamma).\end{equation}
The properties of the spectrum of $\Gamma$ are called the \textit{spectral properties of $\Gamma$} and the study of them is the main topic of spectral graph theory (see for example \cite{spectrabook, crs}). It is well known in the literature that geometric properties on $\Gamma$ can be deduced from analytic properies of $\mathrm{spec}(\Gamma)$; just to give an idea,  Sachs Theorem (see  \cite{dragan1} for a recent formulation of Dress and Stevanovi\'c) shows that bipartite  graphs are characterized by a symmetric spectrum of the adjacency matrix  with respect to zero. We invite the reader to look at the references \cite{spectrabook, crs} for other classical results in spectral graph theory.
\\
Let's recall some other well known notions in graph theory, which are useful to our scopes. If we have two graphs $\Gamma_1 = (V_{\Gamma_1}, E_{\Gamma_1})$ and $\Gamma_2 = (V_{\Gamma_2}, E_{\Gamma_2})$, there are many applications which we may consider, but we are going to introduce a notion of morphism which is reported in \cite{larrion}.
\begin{defn}[See \cite{larrion}]\label{morphism}
A \textbf{morphism of graphs} $f$ is a function $f:V_{\Gamma_1} \longrightarrow V_{\Gamma_2}$ such that for each $v \sim w$ in $V_{\Gamma_1}$ we have $f(v) \sim f(w)$ or $f(v) = f(w)$. 
\end{defn}
In Definition \ref{morphism}, we may abuse the notation and 
write $f: \Gamma_1 \longrightarrow \Gamma_2$, keeping in mind that $f$ is formally defined on the vertices only. Note that Definition \ref{morphism} as most of the notions which we are going to recall in Section \ref{section2} are typical of topological graph theory. It is useful to mention that similar notions can be found on some works of Pisanski and others \cite{Pisanski, lessico}.

\begin{rem}
Usually, the term \textit{morphism of graphs} denotes a function $f: V_{\Gamma_1} \longrightarrow V_{\Gamma_2}$ such that for each $v \sim w$ in $V_{\Gamma_1}$ then $f(v) \sim f(w)$, not allowing the case $f(v) = f(w)$. The Definition \ref{morphism} is more general, and if a function $f$ is a morphism of graphs in the classical sense, then it will be said that the function \textit{preserves the edges}.
\end{rem}

Given a morphism  $f: \Gamma_1 \longrightarrow \Gamma_2$ of graphs, the \textit{image} of $f$ is the graph $f(\Gamma_1)$ whose vertices are given by $f(V_{\Gamma_1})$ and  edges by $E_{\Gamma_2}$, that is,   by the elements $\{f(v), f(w)\}$ such that $v \sim w$ in $\Gamma_1$. Given a subset $S$ of $V_\Gamma$, the \textit{induced subgraph} of $S$ is the graph $\Gamma_S = (S, E_S)$, where $E_S$ is the subset of $E_\Gamma$ given by the couples $\{x,y\}$ such that $x$ and $y$ are both vertices in $S$. Many facts of general topology apply to graphs in the present situation: for instance, the composition of two graph morphisms is a graph morphism too. A morphism  $f: \Gamma_1 \longrightarrow \Gamma_2$ is \textit{injective} (resp. \textit{surjective}) if it is injective (resp. surjective) as function $f:V_{\Gamma_1} \longrightarrow V_{\Gamma_2}$. Given a morphism of graphs $f:\Gamma_1 \longrightarrow \Gamma_2$ and given $v \in V_{\Gamma_2}$, the \textit{fiber} $f^{-1}(v)$ is the induced subgraph of $\Gamma_1$ whose vertex set is the preimage of $v$. Moreover $f$ is an \textit{isomorphism of graphs} if it is injective and surjective at the same time. Observe that each isomorphism of graphs preserves the edges. An isomorphism of graphs $f: \Gamma \longrightarrow \Gamma$ is called an \textit{automorphism of} $\Gamma$ and the group of automorphisms of $\Gamma$ is denoted by $\mathrm{Aut}(\Gamma)$.

\subsection{Cartesian product and strong product of graphs}
Let's consider $\Gamma_1 = (V_{\Gamma_1}, E_{\Gamma_1})$ and $\Gamma_2 = (V_{\Gamma_2}, E_{\Gamma_2})$ which are two graphs.  A very common construction in algebraic combinatorics is given by the graph product, see \cite{bddr,  sab}. It is  worth mentioning that the Cartesian product of graphs fits as a special case of NEPS  in the sense of Cvetkovi\'{c} and  Lu\v{c}i\'{c} (see \cite{neps1, neps2}).

\begin{defn}[See \cite{bddr}, Definition 2.9]\label{cproduct}
The \textbf{Cartesian product} $\Gamma_1  \ \square  \ \Gamma_2$ of the two graphs $\Gamma_1$ and $\Gamma_2$ is the graph whose vertex set is $V_{\Gamma_1} \times V_{\Gamma_2}$ and $(x,y) \sim (x',y')$ if and only if $x \sim x'$ and $y=y'$ or $x=x'$ and $y\sim y'$.
\end{defn}
This notion is largely discussed in \cite{handbook}. The symbol $\square$ recalls the Cartesian product of two copies of $K_2$ which is the graph with two vertices and an edge connecting them. However, it is very easy to produce some different examples of Cartesian products. 
\begin{example}
Let $n$ be a  natural number. A \textit{complete} graph with $n$ vertices is the graph $K_n$ where $\vert V_{K_n} \vert = n$ and $x \sim y$ for each $x \neq y$ in $V_{K_n}$. The graph at the top left in Figure \ref{products} offers an example of  Cartesian product $K_2  \ \square \  K_3$.
\end{example}
One of the best properties of the Cartesian products (from a spectral point of view) is that the adjacency matrix $A_{\Gamma_1  \ \square \  \Gamma_2}$ is well described via the adjacency matrices of $\Gamma_1$ and $\Gamma_2$.

\begin{defn}[See \cite{belardo1}, \S 2.2] Let $A= (a_{ij})$ and $B = (b_{rs})$ be two real matrices of size $n \times m$ and $p \times q$. The \textit{Kronecker product} $A \otimes B$ of $A$ and $B$ is the $np \times mq$-matrix which is block-wise described as follows: the matrix $A \otimes B$ is a $n \times m$ block matrix, where the $(i,j)$-block is the matrix of size $p \times q$
\begin{equation}
    a_{ij}B = \begin{bmatrix} a_{ij}b_{11} & \dots & a_{ij} b_{1q} \\ \vdots & \ddots &\vdots \\
    a_{ij}b_{p1} & \dots & a_{ij} b_{pq}\end{bmatrix}.
\end{equation}
\end{defn}
We shall also recall from \cite{spectrabook, crs} that given $\Gamma_1$ and $\Gamma_2$, we may fix an order on  $V_{\Gamma_1}$ and another order on $V_{\Gamma_2}$. These orders induce the lexicographical order on $V_{\Gamma_1} \times V_{\Gamma_2} = V_{\Gamma_1 \ \square  \ \Gamma_2}$. If $A_{\Gamma_1}$ and $A_{\Gamma_2}$ are the adjacency matrices with respect to the orders on $V_{\Gamma_1}$ and $V_{\Gamma_2}$, then the adjacency matrix of $\Gamma_1  \ \square  \ \Gamma_2$ with respect to the lexicographical order is given by
\begin{equation}\label{agammaproduct1}
    A_{\Gamma_1 \ \square  \ \Gamma_2} = A_{\Gamma_1} \otimes I_{\vert V_{\Gamma_2}\vert} + I_{\vert V_{\Gamma_1}\vert}\otimes A_{\Gamma_2},
\end{equation}
where $I_{\vert V_{\Gamma_i}\vert}$ is the identity matrix of size $\vert V_{\Gamma_i}\vert$ for $i=1,2$. In the present situation  there is a very simple description of $\mathrm{spec}(\Gamma_1 \ \square \ \Gamma_2)$ via a linear combination. In other words, if  
\begin{equation}
\label{agammafactors1}
\mathrm{spec}(\Gamma_1) =\{\lambda_1, \ldots, \lambda_{\vert V_{\Gamma_1}\vert} \mid  \lambda_1 \geq \dots \geq \lambda_{\vert V_{\Gamma_1}\vert}\} \ \mbox{and} \ \mathrm{spec}(\Gamma_2) =\{\mu_1, \ldots, \mu_{\vert V_{\Gamma_2}\vert} \mid  \mu_1 \geq \dots \geq \mu_{\vert V_{\Gamma_2}\vert}\}
\end{equation} 
then we have
\begin{equation}
\label{agammaspectrum1}\mathrm{spec}(\Gamma_1 \ \square \ \Gamma_2) = \{   \lambda_i + \mu_j \mid  \lambda_i \in \mathrm{spec}(\Gamma_1), \  \mu_j \in \mathrm{spec}(\Gamma_2), \ i,j \in \{1, 2\}\}.
\end{equation}
Of course, what we have just done for two graphs can be extended without problems to finitely many graphs and by induction to countably many graphs as well, see \cite{spectrabook, crs}. Let's go ahead and recall some further notions which will be useful for the proof of our main results later on.
\begin{defn}[See \cite{spectrabook}, \S 1.4.8]\label{sproduct}
The \textbf{strong product} $\Gamma_1  \ \boxtimes \ \Gamma_2$ of $\Gamma_1$ and $\Gamma_2$ is defined as the graph whose vertex set is $V_{\Gamma_1} \times V_{\Gamma_2}$ and $(x,y)$ and $(x',y')$ are adjacent if and only if one of the following two conditions holds: \begin{itemize}
    \item[(1).] $x \sim x'$ and $y=y'$, $x=x'$ and $y\sim y'$;
\item[(2).] $x \sim x'$ and $y \sim y'$.
\end{itemize} 
\end{defn}

\begin{example}Similarly to the Cartesian product $K_2  \ \square \ K_3$, we offer an example of strong product  $K_2 \boxtimes K_3$ in the top right of Figure \ref{products}. It is instructive to compare the two constructions which are very different, even if they arise from a not so different idea of product of two combinatorial structures. \end{example}

\begin{figure}[h]
\begin{tikzpicture}[scale=1 ]
\node[vertex,fill=black] (13) at  (-6, 0.75) [scale=0.8,label=right:$2$] {};
\node[vertex,fill=black] (23) at  (-5, -1) [scale=0.8,label=right:$3$]{};
\node[vertex,fill=black] (33) at  (-7, -1)[scale=0.8,label=left:$1$] {};
\node[vertex,fill=black] (43) at  (-6, 2.5) [scale=0.8,label=right:$5$] {};
\node[vertex,fill=black] (53) at  (-3.5, -2) [scale=0.8,label=right:$6$] {};
\node[vertex,fill=black] (63) at  (-8.5, -2) [scale=0.8,label=left:$4$] {};

\node[vertex,fill=black] (11) at  (1, 0.75) [scale=0.8,label=right:$2$] {};
\node[vertex,fill=black] (21) at  (2, -1) [scale=0.8,label=right:$3$]{};
\node[vertex,fill=black] (31) at  (0, -1)[scale=0.8,label=left:$1$] {};
\node[vertex,fill=black] (41) at  (1, 2.5) [scale=0.8,label=right:$5$] {};
\node[vertex,fill=black] (51) at  (3.5, -2) [scale=0.8,label=right:$6$]{};
\node[vertex,fill=black] (61) at  (-1.5, -2)[scale=0.8,label=left:$4$] {};

\node[vertex,fill=black] (12) at  (-6, -5.25) [scale=0.8,label=above:$2$] {};
\node[vertex,fill=black] (22) at  (-5, -7) [scale=0.8,label=right:$3$]{};
\node[vertex,fill=black] (32) at  (-7, -7)[scale=0.8,label=left:$1$] {};
\node[vertex,fill=black] (42) at  (-6, -3.5) [scale=0.8,label=right:$5$] {};
\node[vertex,fill=black] (52) at  (-3.5, -8) [scale=0.8,label=right:$6$] {};
\node[vertex,fill=black] (62) at  (-8.5, -8) [scale=0.8,label=left:$4$] {};

\node[vertex,fill=black] (14) at  (1, -5.25) [scale=0.8,label=left:$2$] {};
\node[vertex,fill=black] (24) at  (2, -7) [scale=0.8,label=below:$3$]{};
\node[vertex,fill=black] (34) at  (0, -7)[scale=0.8,label=left:$1$] {};
\node[vertex,fill=black] (44) at  (1, -3.5) [scale=0.8,label=right:$5$] {};
\node[vertex,fill=black] (54) at  (3.5, -8) [scale=0.8,label=right:$6$] {};
\node[vertex,fill=black] (64) at  (-1.5, -8) [scale=0.8,label=left:$4$] {};

\draw[edge,thick]  (13) --  (43) node[midway, below left] {};
\draw[edge,thick]  (23) --  (53) node[midway, below left] {};
\draw[edge,thick]  (33) --  (63) node[midway, below left] {};
\draw[edge,thick]  (13) --  (23) node[midway, below left] {};
\draw[edge,thick]  (13) --  (33) node[midway, below left] {};
\draw[edge,thick]  (43) --  (53) node[midway, below left] {};
\draw[edge,thick]  (43) --  (63) node[midway, below left] {};
\draw[edge,thick]  (53) --  (63) node[midway, below left] {};
\draw[edge,thick]  (23) --  (33) node[midway, below left] {};

\draw[edge,thick]  (11) --  (41) node[midway, below left] {};
\draw[edge,thick]  (21) --  (51) node[midway, below left] {};
\draw[edge,thick]  (31) --  (61) node[midway, below left] {};
\draw[edge,thick]  (11) --  (21) node[midway, below left] {};
\draw[edge,thick]  (11) --  (31) node[midway, below left] {};
\draw[edge,thick]  (41) --  (51) node[midway, below left] {};
\draw[edge,thick]  (41) --  (61) node[midway, below left] {};
\draw[edge,thick]  (51) --  (61) node[midway, below left] {};
\draw[edge,thick]  (21) --  (31) node[midway, below left] {};
\draw[edge,thick]  (11) --  (51) node[midway, below left] {};
\draw[edge,thick]  (11) --  (61) node[midway, below left] {};
\draw[edge,thick]  (21) --  (61) node[midway, below left] {};
\draw[edge,thick]  (21) --  (41) node[midway, below left] {};
\draw[edge,thick]  (31) --  (51) node[midway, below left] {};
\draw[edge,thick]  (31) --  (41) node[midway, below left] {};

\draw[edge,thick]  (12) --  (32) node[midway, below left] {};
\draw[edge,thick]  (42) --  (62) node[midway, below left] {};
\draw[edge,thick]  (52) --  (42) node[midway, below left] {};
\draw[edge,thick]  (52) --  (32) node[midway, below left] {};
\draw[edge,thick]  (12) --  (22) node[midway, below left] {};
\draw[edge,thick]  (22) --  (62) node[midway, below left] {};

\draw[edge,thick]  (14) --  (34) node[midway, below left] {};
\draw[edge,thick]  (44) --  (64) node[midway, below left] {};
\draw[edge,thick]  (54) --  (44) node[midway, below left] {};
\draw[edge,thick]  (54) --  (34) node[midway, below left] {};
\draw[edge,thick]  (14) --  (24) node[midway, below left] {};
\draw[edge,thick]  (24) --  (64) node[midway, below left] {};
\draw[edge,thick]  (14) --  (44) node[midway, below left] {};
\draw[edge,thick]  (24) --  (54) node[midway, below left] {};
\draw[edge,thick]  (64) --  (34) node[midway, below left] {};
\end{tikzpicture}\caption{From the top left getting to the bottom right we have the Cartesian produc  $K_2 \square K_3$; the strong product $K_2 \boxtimes K_3$; $C_6$ which is the total space of a $2$-covering over $K_3$; $M_3$ which is the M\"obius ladder with $3$ vertices.}\label{products}
\end{figure}
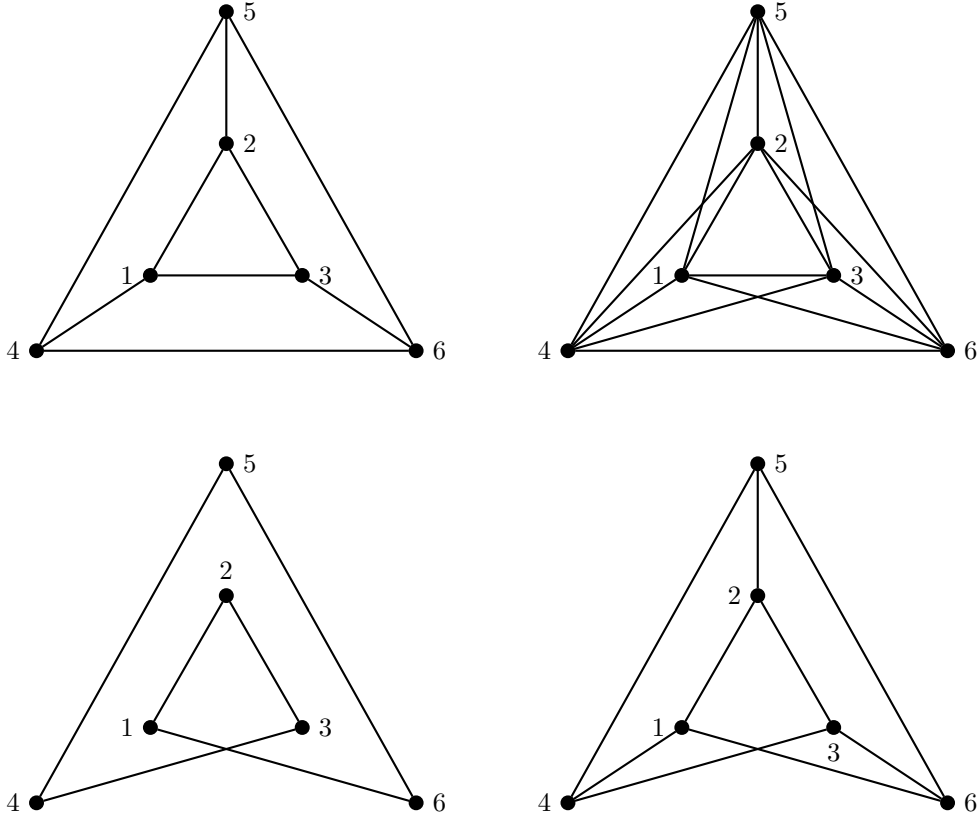

What is the corresponding version of \eqref{agammaproduct1} and \eqref{agammaspectrum1} which we shall formulate when we have two orders of $\Gamma_1$ and $\Gamma_2$ and consider
 $A_{\Gamma_1  \ \boxtimes \  \Gamma_2}$  with respect to the lexicographical order? The answer is summarized by the following formula, which is also known in the literature (see \cite{spectrabook, crs})
\begin{equation}\label{agammaproduct2}
A_{\Gamma_1 \ \boxtimes \ \Gamma_2} = A_{\Gamma_1} \otimes I_{\vert V_{\Gamma_2}\vert} + I_{\vert V_{\Gamma_1}\vert}\otimes A_{\Gamma_2} + A_{\Gamma_1}\otimes A_{\Gamma_2}.
\end{equation}
If we assume that \eqref{agammafactors1} continue to be valid, then \eqref{agammaspectrum1} is no longer a linear combination of the eigenvalues which arise from the spectra of $\Gamma_1$ and $\Gamma_2$, but there is now a mixed term 
\begin{equation}\label{agammaspectrum2}
\mathrm{spec}(\Gamma_1 \   \boxtimes  \ \Gamma_2) = \{  \lambda_i + \mu_j +  \lambda_i\mu_j \mid  \lambda_i \in \mathrm{spec}(\Gamma_1), \  \mu_j \in \mathrm{spec}(\Gamma_2), \ i,j \in \{1, 2\}\}.\end{equation}
Again one can generalize formulas of the above type, but we do not need to go ahead in this direction here. 

\subsection{Covering spaces for graphs} Another fundamental notion in graph theory is the notion of covering graph, see \cite[\S 6.4, \S 14.2.4]{spectrabook} and \cite{gross2}. This can be understood as the combinatorial equivalent for graphs of covering of a topological space. 

Given a vertex $v$ of a graph $\Gamma$, its \textit{neighborhood} $N(x)$ is the graph whose vertex set is $\{v\} \cup A_v$ where $A_v = \{w \in V_\Gamma \mid v \sim w\}$ and the edge set is given by $\{\{v, w\} \mid w \in A_v\}$. Let $p: \Gamma_1 \longrightarrow \Gamma_2$ be a morphism of graphs. The \textit{restriction} of $p$ to a neighborhood $N(x)$ is the graph morphism $p_{\vert x}: N(x) \longrightarrow p(N(x))$. Let $v = p(x)$ and assume that $p(N(x)) = N(v)$ defined as $p_{\vert x}(y) = p(y)$. A \textit{lifting} of $p$ around $v$ is now a graph morphism \begin{equation} l_{v,x} \ : \  N(v) \longrightarrow N(x) \ \ \mbox{such that} \ \ l_{v,x} = p_{\vert x}^{-1}.\end{equation}
Several authors \cite{clarke, spectrabook, crs, gross2} worked on the so called ``covering projections of graphs''.

\begin{defn}[See \cite{gross2}, Chapter 2]\label{kfold}
Let $\widetilde{\Gamma}$ and $\Gamma$ be two graphs with $p: \widetilde{\Gamma} \longrightarrow \Gamma$   surjective morphism of graphs and  $k$  natural number. The triple $(\widetilde{\Gamma}, p, \Gamma)$ is a \textbf{k-fold covering of } $\Gamma$ if the following conditions hold simultaneously: for each vertex $v$ of $\Gamma$
\begin{itemize}
    \item[{\rm (1).}]  $p^{-1}(v)$ contains exactly $k$ vertices,
    \item[{\rm (2).}] for each $x$ in $\widetilde{\Gamma}$ there exists a lifting  $l_{v,x}:N(v) \longrightarrow N(x)$.
\end{itemize}
\end{defn}
\begin{example}\label{exem.cover}
Let $C_n = (V_{C_n}, E_{C_n})$ be the graph with vertices and edges \begin{equation} V_{C_n} = \{1,2, \dots, n\} \  \ \mbox{and} \ \ E_{C_n} = \{\{n,1\}, \{i, i+1\} \vert i=1, \dots, n-1.\}.\end{equation} Note that $C_3 \cong K_3$, while  $C_6$ is depicted on the bottom left of Fig. \ref{products}. We may consider $p: C_6 \longrightarrow C_3$ as the projection which is defined by $p(x) = [x]+1$, where $[x]$ the remainder of the division of $x$ by $3$. The triple $(C_6, p, C_3)$ is an example $2$-graph covering of $C_3$, so of Definition \ref{kfold} for $k=2$. Of course, we may generalize this example to $n \ge 3$ and $k \ge 2$, but it is instructive to  visualize for $n=3$ and $k=2$ in Fig. \ref{products}  the difficulties which we may encounter when we start just with  elementary factors in a product (or in a $k$-fold covering) of two graphs.
\end{example}
In order to understand that Definition \ref{kfold} is a natural variation on the theme of covering space in the context of algebraic topology. First of all, given two topological spaces $\widetilde{X}$ and $X$, a continuous surjective map  $p \colon  \widetilde{X} \to X$ is called  \textit{covering map} if for all $x \in X$, there exists an open neighborhood $U$ of $x$ such that  
\begin{equation}\label{usualcovering1} p^{-1}(U) =\bigcup_{i \in I} U_i, \ \mbox{where} \ U_i \cap U_j =\emptyset \ \forall i \neq j \ \ \mbox{and} \ \ p|U_i : U_i \to  U \ \mbox{is homeomorphism} \ \forall i \in I. \end{equation}
This is a classical notion in homotopy theory and algebraic topology, see \cite[Chapters 17, 19, 20, 21]{kos}. Another excellent reference  is given by \cite[Chapters I, II]{ttd}, where the reader can find an approach which is typical of the  transformation groups. The topological space $\widetilde{X}$ is called \textit{ covering space} for $X$ and has been largely studied in the literature, see \cite[Chapter 10, Appendix 2]{hofmor}. For any point $x$,  \begin{equation}\label{usualcovering2} p^{-1}(x)=\{\widetilde{x} \in \widetilde{X} \ | \ p(\widetilde{x})=x \}\end{equation} is  what one thinks at a \textit{ fiber over x} from the perspective of the algebraic topology and homotopy theory. The special open neighborhoods $U$ of $x$, involved in \eqref{usualcovering1} are said to be \textit{evenly covered neighborhoods} and  form an open cover of  $X$. Indeed, most of the times, these opens arise from a base which is given a priori on the topology of $X$.  The homeomorphic copies (in $\widetilde{X}$) of an evenly covered neighborhood $U$ are  the \textit{sheets}, or \textit{leaves}, over $U$,  see again \cite[Chapters 17, 19, 20, 21]{kos} (or  \cite[Chapters I, II]{ttd} and \cite[Appendix 2]{hofmor}).

 \begin{rem}\label{crucialanalogy} We note that Definition \ref{kfold} is exactly the natural replacement of a covering map, when $\widetilde{X}$ and $X$ are topological spaces. The presence of topologies on $X$ and $\widetilde{X}$ making continuous the surjection $p : \widetilde{X} \to X$ and endowing $X$ of the quotient topology is a significant aspect of the  present discussions. When we move in the category of graphs, we deal with Definition   \ref{morphism} in terms of morphisms. Therefore we search for surjective maps with respect to those in Definition \ref{morphism} and look at fibers in the category of graphs in the sense of Definition \ref{kfold}. Concerning topological spaces, the continuous maps are the corresponding morphisms of the category of topological spaces, hence this is generally different from a morphism in the category of graphs. Note that different morphisms may be present on a same set and these give rise to different structures, see \cite[Appendix 3]{hofmor}.
 \end{rem}
It is instructive to give some examples, which  deal with  algebraic topology. For interested readers it is useful to mention the classical construction of existence of liftings and monodromy theorems for general topological spaces are available in  \cite[Appendix 2]{hofmor}, \cite{ttd},  \cite[Theorem 17.6, Corollary 17.7]{kos}. 
\begin{example}[See \cite{kos}, Exercise 17.9 (c) and (d)]\label{pcv} An example of covering map is given for instance by the power function on the comples numbers without the zero $ z \in \mathbb{C}^* \mapsto p(z)=z^n \in \mathbb{C}^*.$ Here we could think at  $\mathbb{C}^* =\mathbb{R}^2 \setminus \{(0,0)\}$ as the plane with a hole in the origin and with the topology induced  by  the usual euclidean plane $\mathbb{R}^2$. Note also that if we have $p: \widetilde{x}\in \widetilde{X} \mapsto p(\widetilde{x}) \in X$ and $q: \widetilde{y} \in \widetilde{Y} \mapsto q(\widetilde{y}) \in Y$  covering maps, then \begin{equation}p \times q: (\widetilde{x}, \widetilde{y}) \in \widetilde{X} \times \widetilde{Y} \mapsto p \times q  \ (\widetilde{x}, \widetilde{y}) = (p(\widetilde{x}), q(\widetilde{y})) \in X \times Y\end{equation} is a covering map as well. In fact  $\widetilde{X} \times \widetilde{Y} $ is the so-called \textit{product covering space} with factors $\widetilde{X}$ and  $\widetilde{Y}$.  In particular,  if $X=Y$, then $\widetilde{X}=\widetilde{Y}$ and we may consider
\begin{equation}\widetilde{W}=\{(\widetilde{x},\widetilde{y}) \in \widetilde{X} \times \widetilde{X}: p(\widetilde{x})=q(\widetilde{y})\},\end{equation}
noting that also $ (\widetilde{x},\widetilde{y}) \in \widetilde{W} \mapsto p(\widetilde{x}) \in X$ is a covering map. In other words, covering maps can be constructed quite easily on  product covering spaces just via equalizers of $p$ and $q$. In fact we have in this situation  the cartesian product $\widetilde{X} \times \widetilde{X} $ of two copies of the covering space $\widetilde{X}$ with  \textit{equalizer} $\widetilde{W}$, see \cite[Definition A3.43 (ii)]{hofmor} for a more general notion of equalizer in  category theory.
\end{example}
Given a $k$-fold covering $(\widetilde{\Gamma}, p, \Gamma)$, it is possible to identify the vertices of $\widetilde{\Gamma}$ with the elements of $V_\Gamma \times \{1, 2, \dots, k\}$, hence the vertices in  $p^{-1}(v)$ become
\begin{equation}V_{p^{-1}(v)} =\{ (v,i)  \mid i \in \{1, \dots, k\}\}. \end{equation} Therefore for $v,  w \in V_\Gamma$ which are connected in $\Gamma$, it is possible to define the function \begin{equation} \psi_{vw}: p^{-1}(v) \longrightarrow p^{-1}(w) \ \ \mbox{such that}
 \ \     \psi_{vw}(v,i) = l_{v,(v,i)}(w) = (w, \sigma_{vw}(i)) 
\end{equation}   
where $\sigma_{vw}$ is a permutation on the set $\{1,2, \ldots, k\}$. It is folklore  that the set of all permutations $S_k$ on  $\{1,2, \ldots, k\}$ possesses the structure of nonabelian group with respect to the composition of permutations (whenever $k \ge 3$).  If $\vec{E}_\Gamma$ denotes the set of  oriented edges of $\Gamma$, i.e.: 
\begin{equation}
    \vec{E}_\Gamma = \{(v,w) \in V_{\Gamma} \times V_{\Gamma} \mid \{v,w\} \in E_{\Gamma}\},
\end{equation} 
then we may introduce the \textit{voltage function}\begin{equation}
     \sigma \ : \  (v,w) \in \vec{E}_\Gamma \longmapsto \sigma(v,w) = \sigma_{vw} \in S_k,    
\end{equation}
according to the notions which were introduced by Cvetkovi\'{c} and others in \cite{neps1, crs, neps2}. See also \cite[Chapters 2, 5]{gross2} for classical results on the voltage graphs in connection with covering spaces.

Fix an order now on  $V_\Gamma$ and  $\phi \in S_k$.  A generalized version of adjacency matrix  can be formalized here by
\begin{equation}
A_\Gamma(\phi)= \begin{cases}  1, \ \  \mbox{if}  \ \ v \sim w \ \ \mbox{and} \ \ \sigma(v,w) = \phi\\ 
\\
0, \ \ \ \mbox{otherwise}.
\end{cases}
\end{equation} On the other hand, denoting by $\mathrm{perm}(\phi)$ the permutation matrix of $\phi$, we have that the adjacency matrix of $\widetilde{\Gamma}$ is given by
\begin{equation}\label{agammaextra}
    A_{\widetilde{\Gamma}} = \sum\limits_{\phi \in S_k} A_{\Gamma}(\phi) \otimes \mathrm{perm}(\phi).
\end{equation}

\begin{problem}Differently  from what we have seen in  \eqref{agammaspectrum1}  and \eqref{agammaspectrum2} concerning Cartesian products of graphs and strong products of graphs, it is an open problem to find combinatorial formulas which describe $\mathrm{spec}(\widetilde{\Gamma})$ in terms of $\mathrm{spec}(\Gamma)$ when $\widetilde{\Gamma}$, $\Gamma$ and $p$ are arbitrary.
\end{problem}

\subsection{Graph bundles} 
We have already noted that the topological notion of covering can be adapted quite well from the context of the algebraic topology to that of graphs. Then we may go ahead and present the analogue of the notion of fiber bundle, see \cite[Chapter I, \S 5, \S 8, \S 9]{ttd}.  
In order to better understand the notion of graph bundle, let us recall the notion of \textit{fiber bundle}. Given a topological space, an $F$-fiber bundle is a triple $(E,p,B)$ where $E$ and $B$ are topological spaces, $p: E \longrightarrow B$ is a surjection, and for every $b \in B$ there exists a neighborhood $N(b)$ of $b$ and a homeomorphism $f_b: p^{-1}(N(b)) \longrightarrow N(b) \times F$ such that $f_b(p^{-1}(x)) = \{x\} \times F$ for every $x \in N(b)$. The map $f_b$ is called \textit{local trivialization} of the fiber bundle, and its presence ensures the $E$ can be understood as a topological space which locally is the product of the base with the fiber. A classical example of fiber bundle is the M\"obius strip. This is a $[0,1]$-fiber bundle over the circle $S^1$ and, indeed, given an arch $A$ strictly contained in $S^1$, then the restriction of the M\"obius strip to the preimage of $A$ is homeomorphic to $A \times [0,1]$.
\\
\\Let's see what happens when we reformulate the above notions for graphs. We start with a graph $X= (V_X, E_X)$ and a surjective morphism $p: X \longrightarrow \Gamma$ and it is possible to define two graphs which have the same vertex set $V_X$: the first  may be denoted by $R$ and is the graph $(V_{X}, E_R)$ with edges  \begin{equation}E_R = \{ \ \{v,w\} \in E_X \ \mid  p(v) = p(w) \ \ \mbox{for} \  \ v,w \in V_X \}.\end{equation} The second is the graph $\widetilde{X} = (V_{X}, E_{\widetilde{X}})$ with edges \begin{equation} E_{\widetilde{X}} = E_X \setminus E_R.
\end{equation} By default  we get an induced map
\begin{equation} 
\widetilde{p}  \ :  \ \widetilde{X} \longrightarrow \Gamma  \ \ \mbox{such that} \ \ \widetilde{p}(x) = p(x) 
\end{equation} 
which is defined as morphism of graphs by  for each  $x \in V_{\widetilde{X}}$.

\begin{defn}[See \cite{Pisanski}, p.14]\label{graphbundle}
Let $F$ be a graph. An \textbf{$F$-graph bundle} is a triple $(X, p, \Gamma)$, where $X$ and $\Gamma$ are two graphs and $p: X \longrightarrow \Gamma$ is a surjective morphism of graphs such that:
\begin{enumerate}
    \item[{\rm (1).}] For each $v$ in $\Gamma$ the fiber $p^{-1}(v)$ is isomorphic to $F$,
    \item[{\rm (2).}] $\widetilde{p}: \widetilde{X} \longrightarrow \Gamma$ is a $\vert V_F \vert$-fold covering,
    \item[{\rm (3).}] For each $v$ and $w$, the bijection $\psi_{vw}: p^{-1}(v) \longrightarrow p^{-1}(w)$ induced by the covering $(\widetilde{X}, \widetilde{p}, \Gamma)$ is an isomorphism of graphs.
\end{enumerate}
Given a $F$-graph bundle $(X, p, \Gamma)$, the graph $X$ is called the \textbf{total space} of the bundle, $\Gamma$   \textbf{base space}, $F$  \textbf{fiber} and $p$  \textbf{projection} of the bundle. 
\end{defn}
Analogies and differences with the algebraic topology which we have just recalled from \cite{ttd, hofmor, kos} become more clear and transparent after Remark \ref{local.trivia}.

\begin{example}\label{exem.bundle}
Given two graphs $\Gamma_1$ and $\Gamma_2$, the triple $(\Gamma_1  \ \square  \ \Gamma_2, p, \Gamma_i)$ is a $\Gamma_j$-graph bundle, where $i,j = 1,2$, $j \neq i$ and $p(x_1, x_2) = x_i$. So the graph $K_2 \ \square  \  K_3$ in Fig. \ref{products} is also an example of graph bundle over $K_3$, according to Definition \ref{graphbundle}.

Moreover  a $k$-fold covering $(\widetilde{\Gamma}, p, \Gamma)$ can be seen as a graph bundle: in particular it is a graph bundle whose fiber is a graph with $k$ vertices and no edges, so we have additional examples for Definition \ref{graphbundle}. Finally we mention further examples of graph bundles such as the total space of a $K_2$-graph bundle over $C_3 = K_3$; this is the M\"obius ladder $M_3$ which is depicted at the bottom right of Fig. \ref{products}. Indeed, let $q : M_3 \longrightarrow C_3$ be the projection $q(x)=[x]+1$, where $[x]$ the remainder of the division of $x$ by $3$. The triple $(M_3, q, C_3)$ is a graph bundle.
\end{example}

\begin{rem}\label{local.trivia}
In the literature (see \cite{banic}) an equivalent definition of Definition \ref{graphbundle} is given: this is identical to the definition given in this manuscript, but points 2 and 3 are replaced by the condition that for all $v, w \in V_\Gamma$ which are connected, one has also that
\begin{equation}\label{local-trivializ}
   p^{-1}(vw) \cong K_2 \ \square  \  F, 
\end{equation}
where $p^{-1}(vw)$ is the subgraph of $X$ induced by the preimages of $v$ and $w$. This alternative definition highlights the local structure of Cartesian product in a graph bundle and it clearly recalls the local trivializations in the definition of a fiber bundle. In fact, thanks to  \eqref{local-trivializ}, it is easy to show that for each vertex $x$ in $\Gamma$ the graph $p^{-1}(N(x))$, which is the subgraph induced by the preimage of $N(x)$, is isomorphic to $N(x) \square F$ via a map $\phi_b$ which preserves the fibers of $p$ and of $\mathrm{pr} : N(x) \square F \longrightarrow N(x)$ where $\mathrm{pr}(y,w) = y$ is the projection on the first component. Definition \ref{graphbundle} is sometimes named as a notion of  \textit{Cartesian graph bundle}, in opposition with some other type of graph bundles which may be modeled on different type of products. For example there are the notions of \textit{strong graph bundle} in \cite{larrion}, \textit{tensor graph bundles} \cite{tensore}, lexicographic graph bundles \cite{lessico}, and so on, just by replacing the graph product in Equality \ref{local-trivializ}.
\end{rem}
We now introduce an equivalence relation in the space of bundles with the same fiber $F$ and same base $\Gamma$.
\begin{defn}\label{equiv.bundle}
Two $F$-graph bundles $(X, p, \Gamma)$ and $(Y, q, \Gamma)$ are \textbf{equivalent} if there is an isomorphism $\Phi: X \longrightarrow Y$ such that the following diagram commute
\begin{equation}
\adjustbox{scale=1.5, center}{
\begin{tikzcd}
X \arrow{r}{\Phi} \arrow[swap]{d}{p} & Y \arrow{d}{q} \\
\Gamma \arrow{r}{\mathrm{id}}& \Gamma.
\end{tikzcd}}
\end{equation}
A graph bundle $(X, p, \Gamma)$ is called \textbf{trivial} if it is equivalent to the graph bundle $(\Gamma  \ \square  \ F, \pi, \Gamma)$, where $\pi(v,f) = v$ for each $(v,f)$ in $V_{\Gamma} \times V_F$.
\end{defn}
\begin{rem}
 Definition \ref{equiv.bundle} can be seen as the combinatorial version of the notion of isomorphism of fiber bundle \cite{ttd}. We prefer to use the word \textit{equivalent} instead of \textit{isomorphic} since it already exists a notion of isomorphism of vector bundle (see \cite{kwak} for example). The definition of isomorphism of graph bundles is essentially the Definition \ref{equiv.bundle} in which the identity morphism can be replaced by a generic automorphism $\psi$ in $\mathrm{Aut}(\Gamma)$. It follows that two equivalent graph bundles are isomorphic, but the vice-versa does not hold.
\end{rem}

\begin{example}\label{exem.trivia.}
Let $C_6 \ \square  \ K_2$ and $M_{6,2}$ be the two graphs in Fig. \ref{equiv.bundles}. These are the total spaces of two $K_2$-graph bundles $(C_6  \ \square \  K_2, r, C_6)$ and $(M_{6,2}, r, C_6)$ where $r(x) = x$ if $x \leq 6$ and $r(x) = x-6$ otherwise. These two graph bundles are equivalent and the required isomorphism $\phi: C_6 \ \square  \  K_2 \longrightarrow M_{6,2}$ is the permutation in $S_{12}$ defined by the product of the three disjoint transpositions \begin{equation}\phi = (3  \ \ 9) \ \ (4 \ \ 10) \ \ (5 \ \ 11).\end{equation}

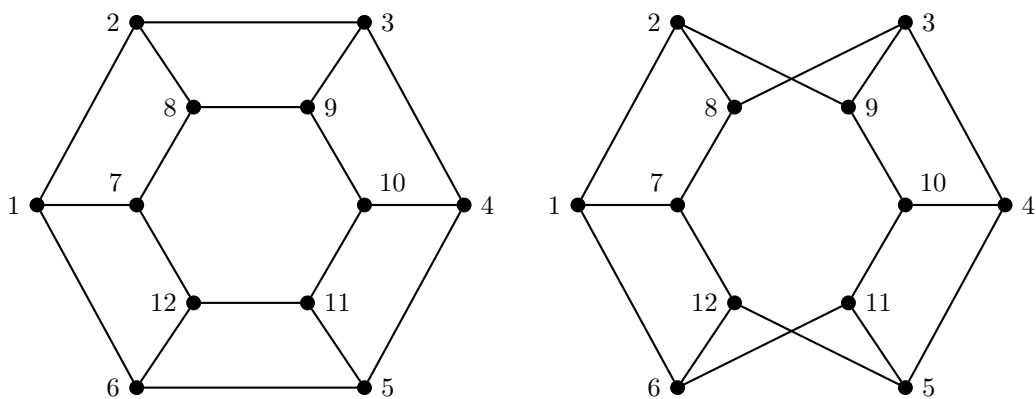
\begin{figure}[h]
\begin{tikzpicture}[scale=0.75]
\node[vertex,fill=black] (1) at  (-5,0)[scale=0.8,label=above left:$7$] {};
\node[vertex,fill=black] (11) at  (-4, 1.732) [scale=0.8,label=left:$8$] {};
\node[vertex,fill=black] (21) at  (-2, 1.732) [scale=0.8,label=right:$9$]{};
\node[vertex,fill=black] (31) at  (-1, 0)[scale=0.8,label=above right:$10$] {};
\node[vertex,fill=black] (41) at  (-2, -1.732)[scale=0.8,label=right:$11$] {};
\node[vertex,fill=black] (51) at  (-4, -1.732)[scale=0.8,label=left:$12$] {};
\node[vertex,fill=black] (61) at  (-6.75,0)[scale=0.8,label=left:$1$] {};
\node[vertex,fill=black] (71) at  (-5,3.232)[scale=0.8,label=left:$2$] {};
\node[vertex,fill=black] (81) at  (-1, 3.232) [scale=0.8,label=right:$3$]{};
\node[vertex,fill=black] (91) at  (0.75, 0)[scale=0.8,label=right:$4$] {};
\node[vertex,fill=black] (101) at  (-1, -3.232)[scale=0.8,label=right:$5$] {};
\node[vertex,fill=black] (111) at  (-5, -3.232)[scale=0.8,label=left:$6$] {};

\node[vertex,fill=black] (2) at  (4.5,0)[scale=0.8,label=above left:$7$] {};
\node[vertex,fill=black] (22) at  (5.5, 1.732) [scale=0.8,label=left:$8$] {};
\node[vertex,fill=black] (23) at  (7.5, 1.732) [scale=0.8,label=right:$9$]{};
\node[vertex,fill=black] (24) at  (8.5, 0)[scale=0.8,label=above right:$10$] {};
\node[vertex,fill=black] (25) at  (7.5, -1.732)[scale=0.8,label=right:$11$] {};
\node[vertex,fill=black] (26) at  (5.5, -1.732)[scale=0.8,label=left:$12$] {};
\node[vertex,fill=black] (27) at  (2.75,0)[scale=0.8,label=left:$1$] {};
\node[vertex,fill=black] (28) at  (4.5, 3.232) [scale=0.8,label=left:$2$] {};
\node[vertex,fill=black] (29) at  (8.5, 3.232) [scale=0.8,label=right:$3$]{};
\node[vertex,fill=black] (210) at  (10.25, 0)[scale=0.8,label=right:$4$] {};
\node[vertex,fill=black] (211) at  (8.5, -3.232)[scale=0.8,label=right:$5$] {};
\node[vertex,fill=black] (212) at  (4.5, -3.232)[scale=0.8,label=left:$6$] {};

\draw[edge, thick]  (1) --  (11) node[midway, above right]  {};
\draw[edge, thick]  (11) --  (21) node[midway, below right] {};
\draw[edge,thick]  (21) --  (31) node[midway, above left] {};
\draw[edge, thick]  (31) --  (41) node[midway, above left] {};
\draw[edge, thick]  (41) --  (51) node[midway, above left] {};
\draw[edge, thick]  (51) --  (1) node[midway, above left] {};
\draw[edge, thick]  (71) --  (81) node[midway, above right]  {};
\draw[edge, thick]  (81) --  (91) node[midway, below right] {};
\draw[edge,thick]  (91) --  (101) node[midway, above left] {};
\draw[edge, thick]  (101) --  (111) node[midway, above left] {};
\draw[edge, thick]  (61) --  (71) node[midway, above left] {};
\draw[edge, thick]  (111) --  (61) node[midway, above left] {};
\draw[edge, thick]  (1) --  (61) node[midway, above right]  {};
\draw[edge, thick]  (11) --  (71) node[midway, above right]  {};
\draw[edge, thick]  (21) --  (81) node[midway, above right]  {};
\draw[edge, thick]  (31) --  (91) node[midway, above right]  {};
\draw[edge, thick]  (41) --  (101) node[midway, above right]  {};
\draw[edge, thick]  (51) --  (111) node[midway, above right]  {};

\draw[edge, thick]  (2) --  (22) node[midway, above right]  {};
\draw[edge, thick]  (28) --  (23) node[midway, below right] {};
\draw[edge,thick]  (23) --  (24) node[midway, above left] {};
\draw[edge, thick]  (24) --  (25) node[midway, above left] {};
\draw[edge, thick]  (25) --  (212) node[midway, above left] {};
\draw[edge, thick]  (26) --  (2) node[midway, above left] {};
\draw[edge, thick]  (27) --  (28) node[midway, above right]  {};
\draw[edge, thick]  (22) --  (29) node[midway, below right] {};
\draw[edge,thick]  (29) --  (210) node[midway, above left] {};
\draw[edge, thick]  (210) --  (211) node[midway, above left] {};
\draw[edge, thick]  (211) --  (26) node[midway, above left] {};
\draw[edge, thick]  (212) --  (27) node[midway, above left] {};
\draw[edge, thick]  (2) --  (27) node[midway, above left] {};
\draw[edge, thick]  (22) --  (28) node[midway, above left] {};
\draw[edge, thick]  (23) --  (29) node[midway, above right]  {};
\draw[edge, thick]  (24) --  (210) node[midway, above right]  {};
\draw[edge, thick]  (25) --  (211) node[midway, above right]  {};
\draw[edge, thick]  (26) --  (212) node[midway, above right]  {};
\end{tikzpicture}\caption{The total spaces $C_6 \ \square  \ K_2$ and $M_{6,2}$ of two equivalent graph $K_2$-bundles over $C_6$.}\label{equiv.bundles}
\end{figure}
\end{example}
\begin{rem}\label{remtree}
 From \cite[Corollary 4]{kwak} we note that each graph bundle $(X, p, T)$ is trivial if $T$ is a tree. Recall that a \textit{tree} is a graph such that for each couple of vertices $v$ and $w$ there is a unique ordered list of distinct vertices $vv_1\dots v_n w$ such that two consecutive vertices in the list are adjacent. On the other hand, there are examples of trivial graph bundles with base different from a tree.
\end{rem}
It is possible to understand the notion of graph bundle in terms of the so-called $F$-voltage functions \cite{Pisanski}. In the next remark, we will introduce this alternative interpretation.

\begin{rem}\label{exem0}
Given two graphs $\Gamma$ and $F$, an $F$\textit{-voltage function} $\phi$ on $\Gamma$ is a voltage function $\phi: \vec{E}_\Gamma \longrightarrow \mathrm{Aut}(F)$ such that $\phi(v,w) = \phi_{vw}$ and $\phi_{wv}=\phi_{vw}^{-1}$ for each oriented edge $(v,w)$. If there is an order on the vertices of $F$, a $F$-voltage map can be seen as a voltage map with values on $\mathrm{Aut}(F) \subseteq S_{\vert V_F \vert}$.
As proved in \cite{Pisanski}, it is possible to construct an $F$-graph bundle with $\Gamma$ as base space in the following way:  first we consider  $\Gamma \times_\phi F$ as the graph whose vertex set is $V_\Gamma \times V_F$ and $(v_1, f_1) \sim (v_2, f_2)$ if and only if $v_1 \sim v_2$ and $f_2 = \phi(v_1, v_2)f_1$ or $v_1 = v_2$ and $f_1 \sim f_2$. Then we
 denote by 
 \begin{equation}\label{fiberproduct1} p: \Gamma \times_{\phi} F \longrightarrow \Gamma  \ \ \mbox{the morphism of graphs such that} \   \  p(v,f) = v  \  \ \mbox{for each} \  (v,f) \in V_\Gamma \times V_F.
 \end{equation} 
 It is easy to check that $p$ is a surjective morphism of graphs and that the triple \begin{equation}\label{fiberproduct2}
 (\Gamma \times_\phi F, p, \Gamma) \ \  \mbox{is an}  \ \ F\mbox{-graph bundle},\end{equation} 
 so Definition \ref{graphbundle} is satisfied.

 It should be also mentioned that it is possible to deduce from Pisanski and others \cite{kwak, Pisanski}  that every $F$-graph bundle $(X, p, \Gamma)$ is equivalent (as per Definition \ref{equiv.bundle} ) to $(\Gamma \times_\phi F, p, \Gamma)$ for some $F$-voltage functions $\phi$. Indeed, fixed for each $v$ in $\Gamma$ an isomorphism $\sigma_v: p^{-1}(v) \longrightarrow F$, then it is sufficient to set $\phi_{vw} = \sigma_w \circ \psi_{vw} \circ \sigma_v^{-1}$ and we obtain an $F$-graph bundle which is equivalent to $(X, p, \Gamma)$. This identification will be useful in order to study the adjacency matrix of the total space of a graph bundle.
\end{rem}

We end with some properties on the adjacency matrix of graph bundles.  As usual we are considering $(X, p, \Gamma)$ as an  $F$-graph bundle and  $(\Gamma \times_{\phi} F, q, \Gamma)$  graph bundle which is equivalent to $(X, p, \Gamma)$. The adjacency matrix $A_{\Gamma \times_{\phi} F}$ of $\Gamma \times_{\phi} F$ may be  described very well in this situation. Since $X$ and $\Gamma \times_{\phi} F$ are isomorphic,  $A_{\Gamma \times_{\phi} F}$ and $A_X$ are the same matrices, up to choose the correct orders of the vertices. Of course, their spectra are equal. This implies that, for the purpose of finding spectral properties of graph bundles, one can consider without loss of generality only graph bundles in the form $\Gamma \times_{\phi} F$.

\begin{lem}[See \cite{kwakk}, Theorem 1]\label{th.kwakk}
Let $(\Gamma \times_{\phi} F, q, \Gamma)$ be an $F$-graph bundle. For each $\psi$ in $\mathrm{Aut}(F)$, denote by $\mathrm{perm}(\psi)$ the permutation matrix of $\psi$. Then the adjacency matrix of $\Gamma \times_{\phi} F$ is given by the formula
\begin{equation}\label{agammasuperextra}
A_{\Gamma \times_{\phi} F} = \sum\limits_{\psi \in \mathrm{Aut}(F)} A_{\Gamma}(\psi) \otimes \mathrm{perm}(\psi) + I_{\vert V_{\Gamma} \vert} \otimes A_{F},
\end{equation}
where $A_{\Gamma}(\psi)$ is the matrix whose entries are $1$ if $\phi_{ij}=\psi$ and $0$ otherwise.
\end{lem}

It is useful to compare \eqref{agammaproduct1}, \eqref{agammaproduct2} and \eqref{agammaextra} with  \eqref{agammasuperextra}, since one can find that (under appropriate assumptions) \eqref{agammasuperextra} works well as generalization of
\eqref{agammaproduct1}, \eqref{agammaproduct2} and \eqref{agammaextra}.

\section{Pullback of graph bundles}\label{section3}

The notion of pullback originates from algebraic topology and differential geometry in a similar way as products of objects in a category, so it is possible to introduce (for arbitrary categories with products and coproducts, see \cite[Appendix 3]{hofmor} for terminology) the notion of pullback, see \cite[Definition A3.43]{hofmor}.
We begin with two graphs $\Gamma$ and $\Gamma'$ along with  an $F$-graph bundle  $(X, p, \Gamma')$. Fixing a morphism of graphs $f: \Gamma \longrightarrow \Gamma'$,  we may introduce an $F$-graph bundle over $\Gamma$ induced by $(X, p, \Gamma')$ and $f$.

\begin{defn}\label{pullback}
The \textbf{pullback along } $f$ \textbf{of } $(X, p, \Gamma')$ is the $F$-graph bundle $(f^*X, f^*p, \Gamma)$ where $f^*X$ is the graph whose vertex set is 
\begin{equation*}
    V_{f^* X} = \{(v, x) \in V_{\Gamma} \times V_{X} \mid f(v) = p(x)\}
\end{equation*}
and $(v,x) \sim (w, y)$ if and only if one of the following three conditions hold
\begin{itemize}
    \item[\rm (1).] $v = w$, and $x\sim y$ (type I edges),
    \item[\rm (2).] $v\sim w$, $f(v) = f(w)$ and $x = y$ (type II edges),
    \item[\rm (3).] $v\sim w$, $f(v) \sim f(w)$ and $x \sim y$ (type III edges).
\end{itemize}
Finally $f^*p: f^*X \longrightarrow \Gamma$ is the morphism such that for each vertex $(v,x)$ in $f^*X$ is defined as $f^*p(v,x) = v$. We denote this graph bundle by $f^*(X, p, \Gamma')$.
\end{defn}

The following example shows a concrete situation where we can visualize a pullback bundle.

\begin{example}
Let $p: C_6 \longrightarrow C_3$ be the map introduced in Example \ref{exem.cover}, if $(M_3, q, C_3)$ is the $K_2$-graph bundle over $C_3$ in Example \ref{exem.bundle}, its pullback bundle along the map $p$ turns out to be the graph bundle $(M_{6,2}, r, C_6)$ of Example \ref{exem.trivia.} and whose total space is depicted in the right of Fig. \ref{equiv.bundles}. The  expression of the surjective homomorphism of graphs $q$ is clear, so we omit the details.
\end{example}

We must now check that Definition \ref{pullback} is appropriate and meaningful.

\begin{prop}
Definition \ref{pullback} is well-posed.
\end{prop}
\begin{proof}
Notice that, given $(x,v) \sim (y,w)$ in $f^*X$, then or $x \sim y$ or $x=y$. This implies that $f^*p$ is a surjective morphism of graphs. It remains to check points (1), (2) and (3) of Definition \ref{graphbundle}.
\begin{enumerate}
 
 \item[(1).] The fiber $f^*p^{-1}(v)$ is the graph with vertices on the form $(v,x)$ where $x$ is in the preimage of $f(v)$ and the edges are the edges of type I connecting vertices on the form $(v,x)$ and $(v,y)$ such that $x \sim y$. This is clearly isomorphic to $p^{-1}(f(v))$ and so to $F$;

 \item[(2).] Let $\widetilde{f^*X}$ be obtained from $f^*X$ by deleting the edges of type I. Then $\widetilde{f^*p}: \widetilde{f^*X} \longrightarrow \Gamma$ is a covering of $\Gamma$. Indeed if we fix $(v,x)$ in $\widetilde{f^*X}$, we note that for each $w \sim v$ in $\Gamma$ there exists exactly one vertex $(w, y_x)$ in $\widetilde{f^*X}$ such that $(v,x) \sim (w, y_x)$: 
 \begin{enumerate} \item[(a).]if $f(v) = f(w)$, then $y_x=x$ (edge of type II), \item[(b).] if $f(v) \sim f(w)$, then $y_x=\psi_{f(v), f(w)}(x)$ (edge of type III), where we must consider the map $\psi_{f(v), f(w)}  :   p^{-1}(f(v)) \to p^{-1}(f(w))$ which is the isomorphism in connection with the oriented edge $(f(v), f(w))$ induced by the structure of $F$-graph bundle of $(X,p, \Gamma')$.  
 Denoting  by $N(v)$ and $N(v,x)$  the two neighborhoods of $v$ and $(v,x)$ and by $l_{v,(v,x)}: N(v) \longrightarrow N(v,x)$ be the morphism which is defined for each vertex $w$ in $N(v)$ by $l_{v,(v,x)}(w) = (w, y_x)$, it turns out that $\widetilde{f^*p}_{\vert N(v,x)}^{-1} =l_{v,(v,x)}$.
 \end{enumerate}
 
 \item[(3).] Let $v$ and $w$ be in $V_\Gamma$ such that $v \sim w$. If $f(v) = f(w)$, then the map which is induced by the  covering $\widetilde{f^*p}: \widetilde{f^*X} \longrightarrow \Gamma$ is 
         $f^*\psi_{v, w} \ : \  f^*p^{-1}(v) \to f^*p^{-1}(w)$ such that $(v, x) \mapsto (w, x).$
 On the other hand, if $f(v) \sim f(w)$ the isomorphism between the fibers $f^*p^{-1}(v)$ and $f^*p^{-1}(w)$ is given by       $f^*\psi_{v, w}  \ : \  f^*p^{-1}(v) \to f^*p^{-1}(w)$ such that   $(v, x) \mapsto (w, \psi_{f(v), f(w)}(x)).$ These maps are clearly isomorphisms of graphs since $\psi_{f(v), f(w)}: p^{-1}(f(v)) \longrightarrow p^{-1}(f(w))$ is an isomorphism of graphs.
\end{enumerate}
\end{proof}

\begin{rem}\label{rem0}
We would like to concentrate on the pullback of graph bundle induced by a voltage function. We would like to highlight that this construction preserves the property of being induced by an $F$-voltage function. Let $f: \Gamma \longrightarrow \Gamma'$ be a graph morphism. Assume that $\phi$ is an $F$-voltage function on $\Gamma'$. Then we define the pullback along $f$ of $\phi$ is the $F$-voltage function $f^*\phi: \vec{E}_\Gamma \longrightarrow \mathrm{Aut}(F)$ defined by
\begin{equation}\label{defn0}
    f^*\phi(v,w) = \begin{cases}
        \mathrm{id}_{F} &\mbox{ if } f(v)=f(w)\\
        \\
        \phi(f(v), f(w)) &\mbox{otherwise.}
    \end{cases}
\end{equation}
Let $(\Gamma' \times_\phi F, p, \Gamma')$ be an $F$-graph bundle induced by a voltage function $\psi$. Let $f: \Gamma \longrightarrow \Gamma'$ be a morphism of graphs. Then
\begin{equation}\label{fiberproduct3}
    f^*(\Gamma' \times_\phi F, p, \Gamma')  \cong (\Gamma \times_{f^*\phi} F, q, \Gamma),
\end{equation}
where $q(v, f) = v$ for each $(v,f)$ in $V_\Gamma \times V_F$. The required isomorphism $\Phi: f^*\Gamma' \times_\phi F \longrightarrow \Gamma \times_{f^*\phi} F$ is simply defined as $\Phi(x, f(x), w) = (x,w)$.
\end{rem}

One needs to be sure that the pullbacks preserve  equivalent graph bundles in order to have no ambiguities. Fortunately, this is the case.

\begin{prop}\label{lemma0}
The pullback preserves the equivalence classes. In other words, let $(X, p, \Gamma')$ and $(Y, q, \Gamma')$ be two equivalent $F$-graph bundles and let $f: \Gamma \longrightarrow \Gamma'$ be a morphism of graphs. Then $f^*(X, p, \Gamma')$ and $f^*(Y, q, \Gamma')$ are equivalent.
\end{prop}
\begin{proof}
Denote by $\Phi: X \longrightarrow Y$ be the isomorphism of graph bundles. Then
$f^*\Phi: f^*X \longrightarrow f^*Y$ is the morphism such that, for each $(v, x)$ in $V_{f^*X}$, $f^*\Phi(v,x) = (v, \Phi(x))$. Notice that, if $(v,x) \sim (v,y)$ are edges of type I, then $x\sim y$ and so, since $\Phi$ is an isomorphism, $f^*\Phi(v,x) = (v, \Phi(x)) \sim (v, \Phi(y)) = f^*\Phi(v,y)$.
If $(v,x) \sim (w,x)$ are edges of type II, then $f^*\Phi(v,x) = (v, \Phi(x)) \sim (w, \Phi(x)) = f^*\Phi(w,x)$. Finally if $(v,x) \sim (w,y)$ are edges of type III, then, again since $\Phi$ is an isomorphism, $f^*\Phi(v,x) = (v, \Phi(x)) \sim (w, \Phi(y)) = f^*\Phi(w,x)$. Clearly $f^*q \circ f^*\Phi = f^*p$.
\end{proof}
\begin{cor}\label{pulltrivial}
The pullback of a trivial bundle is trivial.    
\end{cor}
\begin{proof}
According to Definition \ref{equiv.bundle}, if $(X, p, \Gamma')$ is trivial, then it is equivalent to $(\Gamma' \ \square  \ F, q, \Gamma')$ where $q(x,v) = x$. From Lemma \ref{lemma0}, it is sufficient to prove that $f^*(\Gamma' \ \square  \ F, q, \Gamma') = (\Gamma \ \square  \ F, r, \Gamma)$, where $r(x,g) = x$ for each $(x,g)$ in $V_{\Gamma \ \square  \ F}$. Note that the bundle $(\Gamma' \ \square  \ F, q, \Gamma') = (\Gamma' \times_\phi F, q, \Gamma')$, where $\phi(x,y) = \mathrm{id}_F$ for each $(x,y)$ in $\vec{E}_{\Gamma'}$, this implies  $f^*\phi(v,w) = \mathrm{id}_{F}$ for each $(v,w)$ in $\vec{E}_{\Gamma}$. Thanks to Remark \ref{rem0} we know that
\begin{equation}
    f^*(\Gamma' \ \square  \  F, q, \Gamma') = f^*(\Gamma' \times_\phi F, q, \Gamma') = (\Gamma \times_{f^*\phi} F, f^* q, \Gamma) =  (\Gamma \ \square  \  F, r, \Gamma).
\end{equation}
\end{proof}

\begin{example}\label{exem.pullb}
The $K_2$-graph bundle $(M_3, q, C_3)$  of Example \ref{exem.bundle} is equivalent to the graph bundle $(C_3 \times_\phi K_2, r, C_3)$ where $\phi(1,2) = \phi(2,3) = \mathrm{id}_{K_2}$ and $\phi(1,3) = \alpha$, where $\alpha$ is the automorphism of $K_2$ which exchanges the vertices. Let $p : C_6 \longrightarrow C_3$ be the projection of the $2$-fold covering graph of Example \ref{exem.cover}. The $K_2$-voltage function $p^*\phi$ is given by
\begin{equation}
    p^*\phi(x,y) = \begin{cases}
        \alpha & \mbox{   if   } (x,y) = ((2  \ 3) , (5 \ 6)) \\
        \\
        \mathrm{id}_{K_2} & \mbox{ otherwise.}
    \end{cases}
\end{equation}
The $K_2$-graph bundle $(C_6 \times_{p^* \phi} K_2, p^* r, C_6)$ is equivalent to the graph bundles of Example \ref{exem.trivia.}.
\end{example}

As it happens in algebraic topology, differential geometry and topological algebra, it is possible to observe  functorial properties of the pullbacks. The next proposition is a reminiscence of some classical behaviours of functorial properties in \cite[Propositions A3.33, A3.35, A3.36, Theorems A3.52, Lemma A3.59, Theorem A3.60]{hofmor}.

\begin{prop}\label{funct}
Given the  graphs $\Gamma_1$, $\Gamma_2$ and $\Gamma_3$ with graph morphisms $f: \Gamma_1 \longrightarrow \Gamma_2$ and $g: \Gamma_2 \longrightarrow \Gamma_3$ and  an $F$-graph bundle $(X, p, \Gamma_3)$, the pullback of graph bundles satisfies the following functorial properties: 
\begin{enumerate}
    \item[{\rm (1).}] $f^*[g^*(X, p, \Gamma_3)] \cong [g \circ f]^*(X, p, \Gamma_3)$,
    \item[{\rm (2).}] $\mathrm{id}^*(X, p, \Gamma_3) \cong (X, p, \Gamma_3)$.
\end{enumerate}
\end{prop}

\begin{proof}
Assume that $(X, p, \Gamma_1) \cong (\Gamma_1 \times_\phi F, q, \Gamma_1)$ for some $F$-voltage function $\phi$. Thanks to Example \ref{exem0}, we know that such a $\phi$ exists.
As a consequence of Lemma \ref{lemma0} and Remark \ref{rem0}, we know that
\begin{equation}
g^*(X, p, \Gamma_3) = (g^*X, g^*p, \Gamma_2) \cong (g^*(\Gamma_1 \times_\phi F), g^*(q), \Gamma_2) = (\Gamma_2 \times_{g^*\phi} F, r, \Gamma_2),
\end{equation}
where $r: \Gamma_2 \times_{g^*\phi} F \longrightarrow \Gamma_2$ is the morphism such that $r(v, g) = v$ for each $(v,g)$ in $\Gamma_2 \times_{g^*\phi} F$.
Notice that $f^*(g^*\phi) = [g \circ f]^*\phi$: it directly follows from the definition of pullback of a $F$-voltage function described by \eqref{defn0}. This implies that, if we denote by $q: \Gamma_1 \times_{[g \circ f]^*\phi} F \longrightarrow \Gamma_1$ the morphism such that $q(v, x) = v$ for each vertex $(v,x)$ of $\Gamma_1 \times_{[g \circ f]^*\phi} F$,
\begin{equation}
    f^*[g^*(X, p, \Gamma_3)] = (f^*(g^*X), f^*(g^*p), \Gamma_1) \cong (f^*(\Gamma_2 \times_{g^*\phi} F), f^*(g^*p), \Gamma_1)\end{equation}
 \[   = (\Gamma_1 \times_{f^*(g^*\phi)} F, q, \Gamma_1)    = (\Gamma_1 \times_{[g \circ f]^*\phi} F, q, \Gamma_1) = ([g \circ f]^*X, [g \circ f]^*p, \Gamma_1).\]
This proves (1). Concerning (2), note that the morphism induced by the bijection $\Phi: V_{\mathrm{id}^*X} \longrightarrow V_X$ is defined by $\Phi(x,x)=x$  and it is  an isomorphism.
\end{proof}

\begin{cor}
Let $f: \Gamma \longrightarrow \Gamma'$ be a morphism of graphs and let $(X, p, \Gamma)$ be a $F$-graph bundle. If $f(\Gamma)$ is a tree, then $f^*(X, p, \Gamma)$ is trivial.    
\end{cor}
\begin{proof}
Denote by $\overline{f}:\Gamma \longrightarrow f(\Gamma)$ the morphism defined as $\overline{f}(x) = f(x)$ for each $x$ in $V_\Gamma$ and by $i: f(\Gamma) \longrightarrow \Gamma'$ the morphism such that $i(y) = y$ for each $y$ in $f(\Gamma)$. As a consequence of Proposition \ref{funct} with $f = i \circ \overline{f}$, we have that
\begin{equation}
    f^*(X,p, \Gamma) = \overline{f}^*(i^*(X,p, \Gamma)).
\end{equation}
Note that $i^*(X,p, \Gamma)$ is trivial, since $f(\Gamma)$ is a tree (see also Remark \ref{remtree}). From Corollary \ref{pulltrivial},  its pullback along $\overline{f}$ is trivial. This ends the proof.
\end{proof}

The universal property of Definition \ref{pullback} is explained by the following result. This is in line with the categorical approach of looking at pullbacks of graph bundles in analogy with corresponding constructions in algebraic topology, differential geometry and topological algebra.

\begin{thm}\label{thm0}
Let $\Gamma$ and $\Gamma'$ be two graphs and let $f: \Gamma \longrightarrow \Gamma'$ be a morphism of graphs. Let $(X, p, \Gamma')$ be a graph bundle and denote by $f^*(X,p, \Gamma')$ its pullback along $f$. Then there is a graph morphism $\mathcal{F}: f^*(X) \longrightarrow X$ such that the following diagram commute
\begin{equation}
\adjustbox{scale=1.5, center}{
\begin{tikzcd}
f^*(X) \arrow{r}{\mathcal{F}} \arrow[swap]{d}{f^*p} & X \arrow{d}{p} \\%
\Gamma \arrow{r}{f}& \Gamma'.
\end{tikzcd}}   
\end{equation}
\end{thm}
\begin{proof}
Recall that
\begin{equation}
    V_{f^*(X)} = \{(v,x) \in V_{\Gamma} \times V_X \vert f(v) = p(x)\}.
\end{equation}
Then define $\mathcal{F}: V_{f^*(X)} \longrightarrow X$ as the function such that, for each $(v,x)$ in $V_{f^*(X)}$, $\mathcal{F}(v,x) = x$. First we have to prove that $\mathcal{F}$ is a graph morphism. Let $(v,x) \sim (w,y)$ in $V_{f^*(X)}$. 
\begin{itemize}
\item[(1).] if $(v,x)$ and $(w,y)$ are vertices of an edge of type I, then $v=w$ and $x \sim y$. Then $\mathcal{F}(v,x) = x$ and $\mathcal{F}(v,y) = y$ are vertices of an edge in $X$,
\item[(2).]  if $(v,x)$ and $(w,y)$ are vertices of an edge of type II, then $v\sim w$, $f(v) \sim f(w)$ and $x = y$ and so $\mathcal{F}(v,x) = x = y =\mathcal{F}(w,y)$,
\item[(3).]  finally, if $(v,x)$ and $(w,y)$ are vertices of an edge of type III, then $v\sim w$, $f(v) \sim f(w)$ and $x \sim y$ and so $\mathcal{F}(v,x) = \mathcal{F}(w,y)$.
\end{itemize}
This implies that $\mathcal{F}$ is a graph morphism. Moreover the diagram commutes, since
\begin{equation}
    p(\mathcal{F}(v,x)) = p(x) = f(v) = f(f^*p(v,x)).
\end{equation}
\end{proof}

\subsection{Adjacency matrix of the pullback of graph bundles} 
As already mentioned in the previous section, since each equivalence class of $F$-graph bundles admits a representative expressible through an $F$-voltage function, in order to compute the adjacency matrix of a pullback bundle, we will consider a graph bundle $(\Gamma'\times_\phi F, p, \Gamma')$ and a morphism of graphs $f:\Gamma \longrightarrow \Gamma'$.

We must invoke the notion \textit{Hadamard product of matrices}, which can be found in \cite[\S 10.5]{spectrabook}: if $A=(a_{ij})$ and $B=(b_{ij})$ are two matrices of  same size with $i\in \{1, 2, \ldots,n\}$ and $j \in \{1,2, \ldots, m\}$ for $n,m$ given natural numbers,  the Hadamard product \begin{equation}A \circ B = (c_{ij}) \end{equation} is the matrix with the same size of $A$ and $B$ such that $c_{ij} = a_{ij}b_{ij}$.

Let $M_f$ be the matrix of size $\vert V_{\Gamma'} \vert \times \vert V_{\Gamma}\vert$ defined by
\begin{equation}\label{relevantmatrix}
(M_{f})_{w,v} = \begin{cases}
1 \mbox{ if } f(v) = w,\\
\\
0 \mbox{ otherwise.}
\end{cases}
\end{equation}
\begin{example}
Let's take again the $2$-fold covering of Example \ref{exem.cover}. If we fix the order, which is induced by the well order of the natural numbers, on  $V_{C_6}$ and $V_{C_2}$, then we obtain
\begin{equation}
M_{p} = \begin{bmatrix} 0 & 0& 1& 0 & 0& 1 \\
1 & 0& 0& 1 & 0& 0 \\
0 & 1& 0& 0 & 1& 0
\end{bmatrix}.
\end{equation}
\end{example}

Following the usual matrix notation \cite{spectrabook, crs, gross2},  $M^T_{f}$ denotes the transpose of \eqref{relevantmatrix}. Note that the $(v,w)$-entry of $M^T_fM_f$ satisfies the following equality
\begin{equation}\label{transpose}
    (M^T_fM_f)_{vw}\begin{cases} 1 & f(v)=f(w)\\
    0 &\mbox{otherwise.}
    \end{cases}
\end{equation}

The study of \eqref{relevantmatrix} is fundamental for the description of the adjacency matrix of a pullback bundle $f^*(X,p,\Gamma)$ in terms of $(X, p, \Gamma)$ and $f$.
\begin{thm}[Adjacency Matrix for Pullback of Graph Bundles]\label{1stmaintheorem}
Given a morphism of graphs $f: \Gamma \longrightarrow \Gamma'$ and  an $F$-graph bundle $(\Gamma' \times_\phi F, p, \Gamma')$, the adjacency matrix of the total space of the pullback bundle $(\Gamma \times_{f^*\phi} F, f^*p, \Gamma)$ is given by 
\begin{equation}\label{agammasupersuperextra}
A_{\Gamma \times_{f^*\phi} F}    = \sum\limits_{\psi \in \mathrm{Aut}(F)} [A_{\Gamma} \circ B_{f,\psi}] \otimes \mathrm{perm}(\psi) + I_{\vert V_\Gamma\vert} \otimes F
\end{equation}
where 
\begin{equation}
B_{f,\psi} =\begin{cases}
M_f^T(I_{\vert V_\Gamma\vert } + A_{\Gamma'}(\psi))M_f &\mbox{ if } \psi = \mathrm{id}_F\\
\\
M_f^TA_{\Gamma'}(\psi)M_f &\mbox{ otherwise. }
\end{cases}
\end{equation}
\end{thm}
\begin{proof}
Given an automorphism $\psi$ of $F$, denote by $A_\Gamma(\psi)$ the matrix indexed by $V_\Gamma$ such that $(A_\Gamma(\psi))_{vw}=1$ if $f^*\phi(v,w)= \psi$ and $0$ otherwise.
\\Thanks to Lemma \ref{th.kwakk}, it is sufficient to show that, for the bundle $(\Gamma \times_{f^*\phi} F, f^*p, \Gamma)$, the matrix $A_\Gamma(\psi)$ is equal to $A_{\Gamma} \circ B_{f,\psi}$. In particular the entries 
\begin{equation}\label{vweq}
 (A_{\Gamma}\circ B_{f,\psi})_{vw} = (A_\Gamma(\psi)_{vw}
\end{equation}
for each couple of vertices $v,w$ in $V_\Gamma$ and for each $\psi$ in $\mathrm{Aut}(\Gamma)$. Note that if $v$ and $w$ are two not adjacent vertices in $\Gamma$, then $A_{\Gamma}(v,w)$ and $A_\Gamma(\psi)(v,w)$ are both $0$. So, if $v$ and $w$ are not adjacent  \eqref{vweq} holds. Assume that $v\sim w$ and $f(v) \sim f(w)$. In this case, thanks to Remark \ref{transpose}, $(M_f^TM_{f})_{vw} = 0$. Then, for each $\psi$ in $\mathrm{Aut}(F)$,
\begin{equation}
(A_{\Gamma}\circ B_{f,\psi})_{vw} = (M_f^TA_{\Gamma'}(\psi)M_f)_{vw} = (A_{\Gamma'}(\psi))_{f(v)f(w)}
\end{equation}
which is $1$ if $\phi(f(v), f(w)) = \psi$ and $0$ otherwise. On the other hand, the $(v,w)$-entry of $A_\Gamma(\psi)$ is $1$ if $f^*\phi(v, w) =\psi$ and $0$ otherwise. However, since $f(v) \sim f(w)$, then $\phi(f(v), f(w)) = f^*\phi(v,w)$. This implies  \eqref{vweq}. Finally assume $v \sim w$ and $f(v) = f(w)$. In this case $M_f^TA_{\Gamma'}(\psi)M_f(v,w) = 0$ for each $\psi$ in $\mathrm{Aut}(F)$ and, if $\psi \neq \mathrm{id}_F$, then $A_\Gamma(\psi)(v,w) = 0$. So Equality \ref{vweq} holds for each $\psi \neq \mathrm{id}_F$. On the other hand, if $\psi = \mathrm{id}_F$, then
\begin{equation}
(A_{\Gamma}\circ B_{f,\mathrm{id}_F})_{vw} = (M_f^TM_f)_{vw} = 1,
\end{equation}
since $f(v)= f(w)$, and, for the same reason $(A_{\Gamma}(\mathrm{id}_F))_{vw} = 1$: it directly follows from the definition of $f^*\phi$. This ends the proof.
\end{proof}
\begin{rem}
This result is coherent with the Formulas \eqref{agammaproduct1}, \eqref{agammaextra} and \eqref{agammasuperextra} if $f$ is the identity morphism.
\end{rem}
\begin{example}
Let $p: C_6 \longrightarrow C_3$ be the projection of the $2$-fold covering of $C_3$  in Example \ref{exem.cover} and let $(M_3, q, C_3)$ be the graph bundle in Example \ref{exem.bundle}. As shown in Example \ref{exem.pullb}, the graph bundle $p^*(M_3, q, C_3)$ is equivalent to the graph bundle $(C_6 \times_{p^* \phi} K_2, p^* q, C_6)$, which is the pullback of $(C_3 \times_\phi K_2, r, C_3)$ which has $\phi(1,2) = \phi(2,3) = \mathrm{id}_{K_2}$ and $\phi(1,3) = \alpha$ with   $\alpha \neq \mathrm{id}_{K_2} \in \mathrm{Aut}(K_2)$.If $A_{C_6}(\alpha)$ denotes the adjacency matrix such that the $(v,w)$-entry is equal to $1$ if $p^*\phi(v,w)= \alpha$ and $0$ otherwise, then
\begin{equation*}A_{C_6} \circ B_{p,\mathrm{id}_{K_2}} = 
\begin{bmatrix}
0 & 1 & 0 &0 &0 &1\\
1 & 0 & 1 &0 &0 &0\\
0 & 1 & 0 &1 &0 &0\\
0 & 0 & 1 &0 &1 &0\\
0 & 0 & 0 &1 &0 &1\\
1 & 0 & 0 &0 &1 &0\\
\end{bmatrix}
\circ 
\begin{bmatrix}
1 & 1 & 1 &1 &1 &1\\
1 & 1 & 0 &1 &1 &0\\
1 & 0 & 1 &1 &0 &1\\
1 & 1 & 1 &1 &1 &1\\
1 & 1 & 0 &1 &1 &0\\
1 & 0 & 1 &1 &0 &1\\
\end{bmatrix}
=
\begin{bmatrix}
0 & 1 & 0 &0 &0 &1\\
1 & 0 & 0 &0 &0 &0\\
0 & 0 & 0 &1 &0 &0\\
0 & 0 & 1 &0 &1 &0\\
0 & 0 & 0 &1 &0 &0\\
1 & 0 & 0 &0 &0 &0\\
\end{bmatrix}
= A_{C_6}(\mathrm{id}_{K_2}),
\end{equation*}
and
\begin{equation*}
A_{C_6} \circ B_{p,\alpha} = \begin{bmatrix}
0 & 1 & 0 &0 &0 &1\\
1 & 0 & 1 &0 &0 &0\\
0 & 1 & 0 &1 &0 &0\\
0 & 0 & 1 &0 &1 &0\\
0 & 0 & 0 &1 &0 &1\\
1 & 0 & 0 &0 &1 &0\\
\end{bmatrix}
\circ 
\begin{bmatrix}
0 & 0 & 0 &0 &0 &0\\
0 & 0 & 1 &0 &0 &1\\
0 & 1 & 0 &0 &1 &0\\
0 & 0 & 0 &0 &0 &0\\
0 & 0 & 1 &0 &0 &1\\
0 & 1 & 0 &0 &1 &0\\
\end{bmatrix}
=
\begin{bmatrix}
0 & 0 & 0 &0 &0 &0\\
0 & 0 & 1 &0 &0 &0\\
0 & 1 & 0 &0 &0 &0\\
0 & 0 & 0 &0 &0 &0\\
0 & 0 & 0 &0 &0 &1\\
0 & 0 & 0 &0 &1 &0\\
\end{bmatrix}
=A_{C_6}(\alpha).
\end{equation*}
From the above computations and from Lemma \ref{th.kwakk}, we have a concrete application of Theorem \ref{1stmaintheorem}. In the present situation,  we  obtain that the adjacency matrix of $C_6 \times_{p^*\phi} K_2$ is 
\begin{equation}
    A_{C_6 \times_{p^*\phi} K_2}    = [A_{C_6} \circ B_{p,\mathrm{id}_{K_2}}] \otimes I_2 + [A_{C_6} \circ B_{p,A}] \otimes \overline{I}_2 + I_6 \otimes I_{2} \otimes \overline{I}_2,
\end{equation}
where $\overline{I}_2 =  \begin{bmatrix}0 & 1 \\ 1 & 0\end{bmatrix}$ is the adjacency matrix of $K_2$.
\end{example}

\section{Subdirect product of graph bundles}\label{section4}

We begin with   an $F_1$-graph bundle $(X_1, q_1, \Gamma)$ and  with an $F_2$-graph bundle $(X_2,q_2, \Gamma)$ on the same base $\Gamma$. Denote by $X_1 \boxplus X_2$ the total space of the pullback graph bundle $p_1^*(X_2, p_2, \Gamma)$. Therefore we have the following commutative diagram:
\begin{equation}
\adjustbox{scale=1.5,center}{
\begin{tikzcd}
    X_1 \boxplus X_2 \arrow{r}{\mathcal{Q}_1} \arrow[swap]{d}{q_1^*q_2} & X_2 \arrow{d}{q_2}\\
X_1 \arrow{r}{q_1}& \Gamma,
\end{tikzcd}}
\end{equation}
where the compositions of the  morphisms of graphs is briefly written 
\begin{equation}
(q_1, q_2) =  q_1 \circ q_1^*q_2   
\end{equation} 
and the map $\mathcal{Q}_1$ is induced naturally from the commutativity of the above diagram. In analogy with algebraic topology, in which we have the operation of direct sums of vector bundles, we are going to show that $X_1 \boxplus X_2$ is the total space of a graph bundle over the graph $\Gamma$.
\begin{thm}[Existence of Subdirect Products for Graph Bundles]\label{2ndmaintheorem}
If we have a $F_i$-graph bundle $(X_i, p_i, \Gamma)$ with same base $\Gamma$ for $i=1,2$, then the triple $(X_1  \ \boxplus  \ X_2, (p_1,p_2), \Gamma)$ is an $F_1 \ \square  \ F_2$-graph bundle.
\end{thm}
\begin{proof}
First, observe that \begin{equation} (p_1, p_2) = p_1 \circ p_1^* p_2 \ :  \ X_1 \boxplus X_2 \longrightarrow \Gamma \end{equation} is a surjective morphism of graphs. Then we split the argument in three steps.
\begin{enumerate}
    \item[Step 1.] Given a vertex $v$ in $\Gamma$, then its preimage is the induced subgraph of $X_1 \boxplus X_2$ with vertices 
    \begin{equation}
        V_w = \{(x,y) \in V_{X_1} \times V_{X_2} \mid p_1(x) = p_2(y) = w\}.
    \end{equation}
    The edges connecting two vertices $(x,y)$ and $(x', y')$ in $V_w$ are edges of type I or  II as per Definition \ref{pullback},  while there are no edges of type III. However, this is exactly Cartesian product $p_1^{-1}(w)  \ \square  \ p_2^{-1}(w)$ of the fibers of $w$, which is  isomorphic to $F_1  \ \square  \ F_2$.
    \item[Step 2.] The map $\widetilde{(p_1,p_2)}: \widetilde{X_1 \boxplus X_2} \longrightarrow \Gamma$, where $\widetilde{X_1 \boxplus X_2}$ is the graph whose vertex set is $V_{X_1 \boxplus X_2}$ and the edges are those of type III in $X_1 \boxplus X_2$, is a graph covering of $\Gamma$. 
    Indeed, given $w= (p_1,p_2)(x,y)$, there are two lifting maps \begin{equation}l_{w,x}: N(w) \longrightarrow N(x) \ \  \mbox{and} \ \ l_{w,y}: N(w) \longrightarrow N(y).\end{equation} Therefore the map $l_{w,(x,y)}$ which is defined by
\begin{equation}l_{w,(x,y)} \ : \ z\in N(w)  \longmapsto (l_{w,x}(z), l_{w,y}(z)) \in N(x,y) \end{equation}
     is a lifting. This is due to the fact that a vertex $(x',y')$ is adjacent to $(x,y)$ if and only if $x \sim x'$, $y \sim y'$ and $p_1(x')=p_2(y')$.
\item[Step 3.] Now fix two vertices $w$ and $z$ in $\Gamma$ such that $w \sim z$. There is an induced bijection \begin{equation}\psi_{w,z} \ :  \ p_1^{-1}(w)  \ \square  \ p_2^{-1}(w) \longrightarrow  p_1^{-1}(z)  \ \square  \ p_2^{-1}(z) \ \mbox{such that}  \ \psi_{w,z}(x,y) = (x',y') = l_{w, (x,y)}(z).\end{equation} From Steps 1 and 2, we have that  
\begin{equation}\label{couple}
\psi_{w,z}(x,y) = (l_{w,x}(z), l_{w,y}(z)) = (\psi^1_{w,z}(x), \psi^2_{w,z}(y)) \ \mbox{where} \ \psi^i_{w,z} \ : \  p_i^{-1}(w) \longmapsto p_i^{-1}(z)
\end{equation}
is the isomorphism induced by the graph bundle $(X_i, p_i, \Gamma)$. 
\end{enumerate}
The proof now follows, when we note  (in the final step) that $\psi_{w,z} = (\psi^1_{w,z}, \psi^2_{w,z})$ is a graph isomorphism between $p_1^{-1}(w)  \ \square \  p_2^{-1}(w) $ and $p_1^{-1}(z)  \ \square  \ p_2^{-1}(z)$. \end{proof}
Mimicking the logic of construction in \cite[Chapter A, \S 19, \S 20]{dh} for groups, we may formulate a similar notion (which is completely new to the best of our knowledge) in the present context of graph bundles.
\begin{defn}\label{sub.dir.prod.}
The \textbf{subdirect product} of $(X_1, p_1, \Gamma)$ and $(X_2, p_2, \Gamma)$ is the graph bundle \begin{equation} (X_1, p_1, \Gamma) \ \boxplus  \ (X_2, p_2, \Gamma)\end{equation} given by $(X_1 \boxplus X_2, (p_1,p_2), \Gamma)$. 
\end{defn}

\begin{example}
Let $(M_3, q, C_3)$ and $(C_6 \square K_2, r, C_6)$ be the two $K_2$-graph bundles of Example \ref{exem.bundle}.  Then the subdirect product $(M_3 \boxplus (C_6 \square K_2), (q,r), C_3)$ is the $(K_2  \ \square  \ K_2)$-graph bundle over $C_3$ whose total space is depicted in Fig.\ref{subdir.M3}.
\begin{figure}
\begin{tikzpicture}[scale=1 ]
\node[vertex,fill=black] (11) at  (-8.75,-1)[scale=0.8,label=left:$1$] {};
\node[vertex,fill=black] (21) at  (-8.75,1)[scale=0.8,label=left:$2$] {};
\node[vertex,fill=black] (31) at  (-6.75,1)[scale=0.8,label=above left:$3$] {};
\node[vertex,fill=black] (41) at  (-6.75,-1)[scale=0.8,label=above left:$4$] {};

\node[vertex,fill=black] (12) at  (-6.5,2.732)[scale=0.8,label=left:$5$] {};
\node[vertex,fill=black] (22) at  (-6.5,4.732)[scale=0.8,label=left:$6$] {};
\node[vertex,fill=black] (32) at  (-4.5,4.732)[scale=0.8,label=above left:$7$] {};
\node[vertex,fill=black] (42) at  (-4.5,2.732)[scale=0.8,label=above left:$8$] {};

\node[vertex,fill=black] (13) at  (-1.5, 2.732) [scale=0.8,label=above right:$9$]{};
\node[vertex,fill=black] (23) at  (-1.5, 4.732) [scale=0.8,label=above right:$10$]{};
\node[vertex,fill=black] (33) at  (0.5, 4.732) [scale=0.8,label=right:$11$]{};
\node[vertex,fill=black] (43) at  (0.5, 2.732) [scale=0.8,label=right:$12$]{};

\node[vertex,fill=black] (14) at  (0.75, -1)[scale=0.8,label=above right:$13$] {};
\node[vertex,fill=black] (24) at  (0.75, 1)[scale=0.8,label=above right:$14$] {};
\node[vertex,fill=black] (34) at  (2.75, 1)[scale=0.8,label=right:$15$] {};
\node[vertex,fill=black] (44) at  (2.75, -1)[scale=0.8,label=right:$16$] {};

\node[vertex,fill=black] (15) at  (-1.5, -4.732)[scale=0.8,label=below right:$17$] {};
\node[vertex,fill=black] (25) at  (-1.5, -2.732)[scale=0.8,label=above left:$18$] {};
\node[vertex,fill=black] (35) at  (0.5, -2.732)[scale=0.8,label=right:$19$] {};
\node[vertex,fill=black] (45) at  (0.5, -4.732)[scale=0.8,label=right:$20$] {};

\node[vertex,fill=black] (16) at  (-6.5, -4.732)[scale=0.8,label=left:$21$] {};
\node[vertex,fill=black] (26) at  (-6.5, -2.732)[scale=0.8,label=left:$22$] {};
\node[vertex,fill=black] (36) at  (-4.5, -2.732)[scale=0.8,label=above right:$23$] {};
\node[vertex,fill=black] (46) at  (-4.5, -4.732)[scale=0.8,label=below left:$24$] {};

\draw[edge, thick]  (11) --  (21) node[midway, above right]  {};
\draw[edge, thick]  (21) --  (31) node[midway, above right]  {};
\draw[edge, thick]  (31) --  (41) node[midway, above right]  {};
\draw[edge, thick]  (41) --  (11) node[midway, above right]  {};

\draw[edge, thick]  (12) --  (22) node[midway, above right]  {};
\draw[edge, thick]  (22) --  (32) node[midway, above right]  {};
\draw[edge, thick]  (32) --  (42) node[midway, above right]  {};
\draw[edge, thick]  (42) --  (12) node[midway, above right]  {};

\draw[edge, thick]  (13) --  (23) node[midway, above right]  {};
\draw[edge, thick]  (23) --  (33) node[midway, above right]  {};
\draw[edge, thick]  (33) --  (43) node[midway, above right]  {};
\draw[edge, thick]  (43) --  (13) node[midway, above right]  {};

\draw[edge, thick]  (14) --  (24) node[midway, above right]  {};
\draw[edge, thick]  (24) --  (34) node[midway, above right]  {};
\draw[edge, thick]  (34) --  (44) node[midway, above right]  {};
\draw[edge, thick]  (44) --  (14) node[midway, above right]  {};

\draw[edge, thick]  (15) --  (25) node[midway, above right]  {};
\draw[edge, thick]  (25) --  (35) node[midway, above right]  {};
\draw[edge, thick]  (35) --  (45) node[midway, above right]  {};
\draw[edge, thick]  (45) --  (15) node[midway, above right]  {};

\draw[edge, thick]  (16) --  (26) node[midway, above right]  {};
\draw[edge, thick]  (26) --  (36) node[midway, above right]  {};
\draw[edge, thick]  (36) --  (46) node[midway, above right]  {};
\draw[edge, thick]  (46) --  (16) node[midway, above right]  {};

\draw[edge, thick]  (11) --  (12) node[midway, above right]  {};
\draw[edge, thick]  (21) --  (22) node[midway, above right]  {};
\draw[edge, thick]  (31) --  (32) node[midway, above right]  {};
\draw[edge, thick]  (41) --  (42) node[midway, above right]  {};

\draw[edge, thick]  (11) --  (16) node[midway, above right]  {};
\draw[edge, thick]  (21) --  (26) node[midway, above right]  {};
\draw[edge, thick]  (31) --  (36) node[midway, above right]  {};
\draw[edge, thick]  (41) --  (46) node[midway, above right]  {};

\draw[edge, thick]  (14) --  (13) node[midway, above right]  {};
\draw[edge, thick]  (24) --  (23) node[midway, above right]  {};
\draw[edge, thick]  (34) --  (33) node[midway, above right]  {};
\draw[edge, thick]  (44) --  (43) node[midway, above right]  {};

\draw[edge, thick]  (12) --  (23) node[midway, above right]  {};
\draw[edge, thick]  (22) --  (13) node[midway, above right]  {};
\draw[edge, thick]  (32) --  (43) node[midway, above right]  {};
\draw[edge, thick]  (42) --  (33) node[midway, above right]  {};

\draw[edge, thick]  (16) --  (25) node[midway, above right]  {};
\draw[edge, thick]  (26) --  (15) node[midway, above right]  {};
\draw[edge, thick]  (36) --  (45) node[midway, above right]  {};
\draw[edge, thick]  (46) --  (35) node[midway, above right]  {};

\draw[edge, thick]  (14) --  (15) node[midway, above right]  {};
\draw[edge, thick]  (24) --  (25) node[midway, above right]  {};
\draw[edge, thick]  (34) --  (35) node[midway, above right]  {};
\draw[edge, thick]  (44) --  (45) node[midway, above right]  {};

\end{tikzpicture}\caption{The graph $M_6 \boxplus (C_6 \square K_2)$.}\label{subdir.M3}
\end{figure}
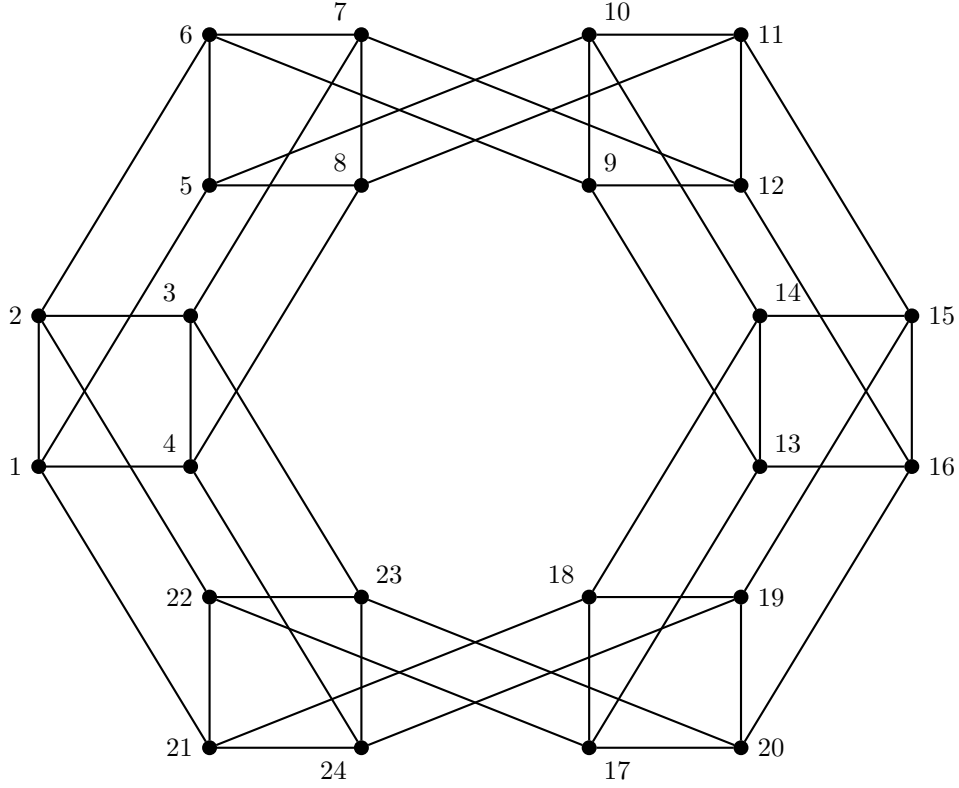
\end{example}

In next result it will be proved the invariance under equivalences for Definition \ref{sub.dir.prod.}. This will be useful, in order to calculate the corresponding adjacency matrix and  spectrum later on.

\begin{lem}\label{stab}
If $(X_i, p_i, \Gamma)$ are  graph bundles for each $i=1,2$ and $(Y_i, q_i, \Gamma)$ graphs bundles which are  equivalent to $(X_i, p_i, \Gamma)$,  then $(X_1  \ \boxplus \ X_2, (p_1,p_2), \Gamma)$ is equivalent to $(Y_1 \  \boxplus  \ Y_2, (q_1,q_2), \Gamma)$.
\end{lem}
\begin{proof}
Denote by $\Phi_i$ be the bundle isomorphism between $X_i$ and $Y_i$. Then the required isomorphism is $\Phi_1 \boxplus \Phi_2: X_1  \ \boxplus  \ X_2 \longrightarrow Y_1  \ \boxplus  \ Y_2$ defined as $\Phi_1 \boxplus \Phi_2(x_1, x_2) = (\Phi_1(x_1), \Phi_2(x_2))$. Indeed it can be easily checked that  $\Phi_1 \boxplus \Phi_2$ preserves the types of edges.   
\end{proof}

We are able to show a result of invariance for subdirect product of graph bundles. In particular it is proved that the pullback commute with the subdirect product.

\begin{thm}\label{sumstab}
Let $(X_i, p_i, \Gamma)$ be two $F_i$-graph bundles over the same graph $\Gamma$ and  $f: \Gamma' \longrightarrow \Gamma$  a morphism of graphs. Then
\begin{equation}
f^*(X_1 \boxplus X_2, (p_1,p_2), \Gamma) \cong f^*(X_1, p_1, \Gamma) \boxplus f^*(X_2, p_2, \Gamma).
\end{equation}
\end{thm}
\begin{proof}
The vertex set of $f^*(X_1 \boxplus X_2)$ is
\begin{equation}
\{(v,x, y) \in V_{\Gamma'}\times V_{X_1} \times V_{X_2} \mid f(v)= p_1(x)= p_2(y)\}.    
\end{equation}
On the other hand, the vertex set of $f^*X_1 \boxplus f^*X_2$ is
\begin{equation}
\{(v,x, v, y) \in V_{\Gamma'} \times V_{X_1} \times V_{\Gamma'} \times V_{X_2} \mid f(v)= p_1(x)= p_2(y)\}.    
\end{equation}
Let $\Phi$ be the map defined for each vertex in $f^*(X_1 \boxplus X_2)$ by \begin{equation}\Phi(v,x, y) = (v,x, v, y).\end{equation} Clearly $\Phi$ commutes with the projections: so, if it is an isomorphism of graphs, then it is also an isomorphism of graph bundles. Since $f^*(X_1  \ \boxplus \  X_2, (p_1,p_2), \Gamma)$ and $f^*(X_1, p_1, \Gamma) \boxplus f^*(X_2, p_2, \Gamma)$ are both $F_1  \ \square  \ F_2$-graph bundle on the same base $\Gamma$, it is sufficient to show that $\Phi$ is a morphism of graph which preserves the edges, indeed two graph bundles with the same fiber and the same base also have the same number of edges. This is a consequence of the fact that a graph bundle is locally a Cartesian product. Let's argue according to the types of the edge as per Definition \ref{pullback}.

\begin{enumerate}
    \item[(1).] 
Let $(v,x, y)\sim(v,x',y')$. These vertices form an edge of type $I$ for $f^*(X_1 \boxplus X_2)$. In particular either $x \sim x'$ and $y=y'$, or $x = x'$ and $y \sim y'$. In the first case $\Phi(v,x, y) = (v,x, v, y)\sim(v,x',v, y')= \Phi(v,x', y')$ is an edge of type II for $f^*(X_1, p_1, \Gamma)  \ \boxplus \ f^*(X_2, p_2, \Gamma)$. In the second case, it is an edge of type I.

    \item[(2).] 
Let $(v,x, y)\sim(w,x,y)$ be an edge of type II. Then $\Phi(v,x, y) = (v,x, v, y)\sim(w,x,w, y)= \Phi(w,x, y)$ is also an edge of type II for $f^*(X_1, p_1, \Gamma)  \ \boxplus  \ f^*(X_2, p_2, \Gamma)$.

    \item[(3).] 
Let $(v,x, y)\sim(w,x',y')$ be an edge of type III. Then we have $f(v) \sim f(w)$ and $(x,y) \sim (x',y')$.  Since $(p_1,p_2)(x,y)= f(v)$ and $(p_1,p_2)(x',y')= f(w)$, we find that $x \sim x'$ and $y\sim y'$. So the vertices $(v,x) \sim (w,x')$ in $f^*X_1$, $(v,y) \sim (w,y')$ in $f^*X_2$ and $\Phi(v,x, y) = (v,x, v, y)\sim(w,x',w, y')= \Phi(w,x', y')$ is an edge of type III in $f^*(X_1, p_1, \Gamma)  \ \boxplus  \ f^*(X_2, p_2, \Gamma)$. 
\end{enumerate}
The result follows.
\end{proof}
Another interesting property which appears in the present investigations is illustrated below.
\begin{prop}[Universal Property of Subdirect Product of Graph Bundles]\label{3rdmaintheorem}
Let $(X_i, p_i, \Gamma)$ be a $F_i$-graph bundle for $i=1,2$. Let $Y$ be a graph and assume the existence of two morphisms $\alpha_i: Y \longrightarrow X_i$ such that $p_1 \circ \alpha_1 = p_2 \circ \alpha_2$ is a morphism which preserves the edges. Then there is a morphism $\alpha_1 \boxplus \alpha_2: Y \longrightarrow X_{1}\boxplus X_2$ such that the following diagram commutes.
\begin{equation}
\adjustbox{scale=1.5,center}{
\begin{tikzpicture}
\node (E) at (0,0) {};
\node[left=of E] (S) {$Y$};
\node[below=of E] (P) {$X_1 \boxplus X_2$};
\node[right=of P] (Q) {$X_2$};
\node[below=of P] (W) {$X_1$};
\node[below=of Q] (G) {$\Gamma$};

\draw[->] (S)--(P) node [midway,right] {$\alpha_1 \boxplus \alpha_2$};
\draw[->] (P)--(Q) node [midway,above] {$\mathcal{Q}_1$};
\draw[->] (P)--(W) node [midway,right] {$q_1^*q_2$};
\draw[->] (W)--(G) node [midway,below ] {$q_1$};
\draw[->] (Q)--(G) node [midway, right] {$q_2$};
\draw[->,bend right] (S) to node [below left] {$\alpha_1$} (W) ;
\draw[->,bend left] (S) to node [above] {$\alpha_2$} (Q) ;
\end{tikzpicture}}
\end{equation}
\end{prop}
\begin{proof}
Let $\alpha: Y \longrightarrow X_1  \ \boxplus  \ X_2$ be the function defined by $\alpha_1   \boxplus   \alpha_2(y) = (\alpha_1(y), \alpha_2(y))$. This function clearly makes the diagram commutative, so it remains to prove that $\alpha_1  \ \boxplus  \ \alpha_2$ is a morphism of graphs. Since $p_1 \circ \alpha_1 = p_2 \circ \alpha_2$ preserves the edges, we have that also $\alpha_1$ and $\alpha_2$ preserve them. This means that, for each $y\sim z$ in $Y$, $\alpha_1(z) \sim \alpha_1(y)$, $\alpha_2(z) \sim \alpha_2(y)$, and $p_1(\alpha_1(z)) \sim p_1(\alpha_1(y))$. This is an edge of type III considering $X_1  \ \boxplus  \ X_2 = p_1^*X_2$.
\end{proof}

In the proof of Theorem \ref{3rdmaintheorem}, it is crucial that $p_1 \circ \alpha_1 = p_2 \circ \alpha_2$ preserves the edges, since otherwise the statement fails. For example, if $\Gamma$ is a point and  $X_1 = X_2 = K_2$, then $X_1  \ \boxplus  \ X_2 = K_2  \ \square  \ K_2$. Assume that $Y = K_2$ and  $\alpha_1 = \alpha_2 = \mathrm{id}_{K_2}$. Then there is no morphism $\alpha_1   \boxplus   \alpha_2 \ :  \ Y \longrightarrow X_1 \  \boxplus \ X_2$ such that $p_1^*p_2 \circ \alpha_1   \boxplus   \alpha_2 = \alpha_1$ and   $p_2^*p_1 \circ \alpha_1   \boxplus   \alpha_2 = \alpha_2.$

\begin{rem}\label{crosssectionnormal} The notion of  \textit{ section} (also called \textit{crossed section} by some authors) is another common notion in topological algebra, differential geometry and algebraic topology. This can be given in general categories: for instance, if we look at the category of topological groups (with corresponding morphisms) and at \cite[Definition 10.9]{hofmor}, then we  consider a topological group $G$ acting on a Hausdorff space $X$ and  a continuous function $\sigma : X/G \to X$ is a section if  $\mathrm{pr}_{G,X} \circ \sigma =\mathrm{id}_{X/G}$, where $\mathrm{id}_{X/G}$ is the identity map of the orbit space $X/G$ and  $\mathrm{pr}_{G,X}$ is the canonical projection of $X$ onto $X/G$. 
From a purely topological point of view, a section of a fibration $(X, p, M)$ is a continuous function $\sigma: U \longrightarrow X$ where $U$ is an open subset of $M$ such that $p \circ \sigma = id_U$.
\\Something similar can be formulated in the context of graph bundles: consider an $F$-graph bundle $(X, p, \Gamma)$ and a \textit{ section}  is a morphism of graphs $\beta: Y \longrightarrow X$, where $Y=(V_Y, E_Y)$ is a graph such that $V_Y \subseteq V_\Gamma$ and $E_Y \subseteq E_\Gamma$ and $p \circ \beta = \mathrm{id}_Y$.

A \textit{global section} is a section such that $Y = \Gamma$. Note that a section  preserves the edges and this means that Theorem \ref{3rdmaintheorem} holds for sections which are defined on the same $Y$. In particular, if $\alpha_1$ and $\alpha_2$ are global sections in Theorem \ref{3rdmaintheorem}, so is $\alpha_1 \boxplus \alpha_2$.
\end{rem}

\subsection{Adjacency matrix of the subdirect product of graphs}
Following the strategy of proof in Lemma \ref{th.kwakk}, given an $F$-graph bundle $(X, p, \Gamma)$, we fix for each vertex $v \in V_\Gamma$ an isomorphism $\sigma_v: p^{-1}(v) \longrightarrow F$ from the fiber to $F$. Therefore for two adjacent vertices $v \sim w$ in $X$, the isomorphism $\psi_{vw}: p^{-1}(v) \longrightarrow p^{-1}(w)$ can be seen as an isomorphism of $\mathrm{Aut}(F)$ up to composing with $\sigma_w$ and $\sigma_v^{-1}$. 
Let $(X_1, p_1, \Gamma)$ and $(X_2, p_2, \Gamma)$ be two graph bundles with fibers $F_1$ and $F_2$. For each $v$ in $V_{\Gamma}$, we fix 
\begin{equation}
(\sigma_v,\sigma_w): p_1^{-1}(v)  \ \square  \ p_2^{-1}(v) \longrightarrow F_1 \  \square \  F_2.
\end{equation}
From  \eqref{couple} in the proof of Theorem \ref{2ndmaintheorem}, we have for two adjacent vertices $v, w \in V_\Gamma$ the isomorphism of graphs \begin{equation}\psi_{vw} = (\psi^1_{vw}, \psi^2_{vw}),\end{equation} where $\psi^i_{vw}$ is the isomorphism in $X_i$ for $i=1,2$. Note that, if we fix some order on $V_{F_i}$ and  the lexicographical order on $V_{F_1  \ \square  \ F_2}$, then we get the permutation matrix
\begin{equation}\label{perm.sp.}
\mathrm{perm}(\psi^1_{vw}, \psi^2_{vw}) = \mathrm{perm}(\psi^1_{vw}) \otimes \mathrm{perm}(\psi^2_{vw}).
\end{equation}
This fact is a consequence of the notion of Kronecker product of permutations (see \cite{spectrabook, crs, wang}).

\begin{thm}[Adjacency Matrix for Subdirect Products of Graph Bundles]\label{4thmaintheorem}
Let $(X_i, p_i, \Gamma)$ be  $F_i$-graph bundles for $i=1,2$ and for each $\psi^i \in \mathrm{Aut}(F_i)$ denote by $\mathrm{perm}(\psi^i)$ the permutation matrix of $\psi^i$. Then 
\begin{equation}
A_{X_1 \boxplus X_2} = {\underset{\psi^2 \in \mathrm{Aut}(F_2)}{\underset{\psi^1 \in \mathrm{Aut}(F_1)}\sum}} A_{\Gamma}(\psi^1, \psi^2) \otimes \mathrm{perm}(\psi^1) \otimes \mathrm{perm}(\psi^2) + I_{\vert V_{\Gamma} \vert} \otimes A_{F_1} \otimes I_{\vert V_{F_2} \vert } + I_{\vert V_{\Gamma} \vert} \otimes I_{\vert V_{F_1} \vert } \otimes A_{F_2}.
\end{equation}
\end{thm}

\begin{proof}
From \eqref{perm.sp.} and Lemma \ref{th.kwakk}, we find that
 \begin{equation}
A_{X_1 \boxplus X_2} = {\underset{\psi^2 \in \mathrm{Aut}(F_2)}{\underset{\psi^1 \in \mathrm{Aut}(F_1)}\sum}} A_{\Gamma}(\psi^1, \psi^2) \otimes \mathrm{perm}(\psi^1) \otimes \mathrm{perm}(\psi^2)+ I_{\vert V_{\Gamma} \vert} \otimes A_{F_1 \square F_2}.
\end{equation}
Then, in order to conclude, it is sufficient to recall that 
\begin{equation}
A_{F_1 \square F_2} = A_{F_1} \otimes I_{\vert V_{F_2} \vert }  + I_{\vert V_{F_1} \vert } \otimes A_{F_2}.
\end{equation}
\end{proof}
\subsection{K-theory from graph bundles}
In next lemmas some properties about the direct sum of graph bundles are proved: in particular it will be proved that the set of all the equivalence classes of graph bundles over $\Gamma$ endowed with the subdirect sum is an \textit{abelian monoid}, i.e. it is commutative, associative and there exists a neutral element.
These properties are the same as those satisfied by vector bundles over topological spaces. We will see how these properties will allow us to define the concept of network $K$-theory.
\begin{thm}[Monoidal Structure of Subdirect Product of Graph Bundles]\label{5thmaintheorem}
Let $(X_i, p_i, \Gamma)$ be an $F_i$-graph bundle with same base $\Gamma$ where $i=1,2,3$ on the same base $\Gamma$. Then the following properties holds
\begin{itemize}
    \item[{\rm (1).}] Commutativity: 
    \begin{equation*}
        (X_1 \boxplus X_2, (p_1,p_2), \Gamma) \cong (X_2 \boxplus X_1, (p_2,p_1), \Gamma),
    \end{equation*}
    \item[{\rm (2).}] Associativity: 
    \begin{equation*}
        ([X_1 \boxplus X_2] \boxplus X_3, ((p_1,p_2), p_3), \Gamma) \cong (X_1 \boxplus [X_2 \boxplus X_3], (p_1,(p_2,p_3)), \Gamma),
    \end{equation*}
    \item[{\rm (3).}] Existence of a neutral element: 
    \begin{equation*}
        (X_1, p_1, \Gamma) \boxplus (\Gamma, id, \Gamma) = (X_1, p_1, \Gamma).
    \end{equation*}
\end{itemize}
\end{thm}
\begin{proof}It will be proved point by point.\\
\\
\textit{Commutativity:} Note that \begin{equation}V_{X_1 \boxplus X_2} = \{(x,y) \in V_{X_1}\times V_{X_2} \mid p_1(x)=p_2(y)\} \ \mbox{and}  \ V_{X_2 \boxplus X_1} = \{(y,x) \in V_{X_2} \times V_{X_1}\mid p_2(y)=p_1(x)\}.\end{equation} The required isomorphism  \begin{equation}
\Phi: X_1 \boxplus X_2 \longrightarrow X_2 \boxplus X_1  \ \ \mbox{is defined by}   \ \ \Phi(x,y) = (y,x).\end{equation} It is easy to show that this map preserves the edges; in particular $\{\Phi(x,y),\Phi(x',y')\}$ is an edge of type I if $\{(x,y), (x', y')\}$ is an edge of type II; it is an edge of type II if $\{(x,y), (x', y')\}$ is an edge of type I; it is an edge of type III otherwise. Since two $F$-graph bundles on the same base space have the same number of edges, then we obtain that $\Phi$ is an isomorphism. Finally $(p_2,p_1) \circ \Phi = (p_1,p_2)$ and so $\Phi$ is a morphism of graph bundles.\\
\\
\textit{Associativity:} Note that $V_{X_1 \boxplus [X_2 \boxplus X_3]} = V_{[X_1 \boxplus X_2] \boxplus X_3}$ and they are both equal to
\begin{equation}
\{(x_1, x_2, x_3) \in V_{X_1} \times V_{X_2} \times V_{X_3} \mid p_1(x_1) = p_2(x_2) = p_3(x_3)\}.
\end{equation}
The map \begin{equation}\label{identity} \mathrm{id} : V_{X_1 \boxplus [X_2 \boxplus X_3]} \longrightarrow V_{[X_1 \boxplus X_2] \boxplus X_3}  \  \ \ \mbox{is defined by} \  \ \ \mathrm{id}(x_1,x_2,x_3) = (x_1,x_2,x_3)\end{equation} and is an isomorphism of graphs. Indeed, if $(x_1, x_2, x_3)\sim (y_1, y_2, y_3)$ in $X_1 \boxplus [X_2 \boxplus X_3]$, 
\begin{enumerate}
\item[(1).] and they form an edge of type I, then $x_1 = y_1$ and $(x_2, x_3)$ and $(y_2, y_3)$ are vertices in the fiber of $p_1(x_1)$ along $(p_2,p_3)$. This is $p_{2}^{-1}(p_1(x_1))  \ \square  \ p_{3}^{-1}(p_1(x_1))$. Then are two alternatives: either $x_2 = y_2$ and $x_3 \sim y_3$, or $x_2 \sim y_2$ and $x_3 = y_3$. In the first case it is an edge of type I for $[X_1 \boxplus X_2] \boxplus X_3$, in the second case it is an edge of type II;

\item[(2).] and they form an edge of type II, then $x_2 = y_2$, $x_3 = y_3$ and $x_1 \sim y_1$ with $p_1(x_1) = p_1(y_1)$. This is an edge of type II for $[X_1 \boxplus X_2] \boxplus X_3$,

\item[(3).] and they form an edge of type III, then $x_1 \sim y_1$, $p_1(x_1) \sim p_1(y_1)$ and $(x_2,x_3) \sim (y_2, y_3)$. Note that $p_1(x_1) \sim p_1(y_1)$ implies that also $p_2(x_2) \sim p_2(y_2)$ and $p_3(x_3) \sim p_3(y_3)$ from the definition of $V_{X_1 \boxplus [X_2 \boxplus X_3]}$. This means that the edge in $X_2 \boxplus X_3$ which connects $(x_2,x_3)$ and $  (y_2, y_3)$ is of type III. Hence $x_2 \sim y_2$ and $x_3 \sim y_3$ and this implies that the edge connecting $(x_1, x_2, x_3)$ and $(y_1, y_2, y_3)$ is of type III in $[X_1 \boxplus X_2] \boxplus X_3$.
\end{enumerate}

In order to conclude that \eqref{identity} is an isomorphism of graphs, it would be necessary to prove that all the edges of $[X_1 \boxplus X_2] \boxplus X_3$ are edges on the image of $id(X_1 \boxplus [X_2 \boxplus X_3])$. However it is sufficient to observe that they have the same number of egdes. Indeed, given an $F$-graph bundle $(X, p, \Gamma)$, then $\vert E_{X} \vert = E_{\Gamma \square F}$. This fact can be deduced by observing that the adjacency matrix of $X$ has the same numbers of entries equal to $1$ of $A_{\Gamma} \otimes I_{\vert F \vert} + I_{\vert V_\Gamma\vert} \otimes A_F$. So this implies that given two bundles on the same base space and with the same fiber, then they have the same number of edges. In particular this happens for $[X_1 \boxplus X_2] \boxplus X_3$ and $X_1 \boxplus [X_2 \boxplus X_3]$. The proof of this point ends since $((p_1,p_2), p_3) \circ \mathrm{id} = (p_1,(p_2, p_3))$.\\
\\
\textit{Existence of a neutral element:} notice that the fibers of the graph bundle $(\Gamma, id_\Gamma, \Gamma)$ are isomorphic to the trivial bundle with just one vertex. This implies that, given a graph bundle $(X, p, \Gamma)$ the map $\Phi: X \ \boxplus \ \Gamma \longrightarrow X$ which is defined by $\Phi(x, v) = x$ for each $(x,v)$ in $V_{X \boxplus \Gamma}$ is the required isomorphism of graph bundles.
\end{proof}
Given  a graph $F$, we may denote the $n$-\textit{times iterated Cartesian product of} $F$ by 
\begin{equation}\label{iteration}
F^n = \underbrace{F \ \square  \  F \ \square  \  F \ \square  \  \ldots \ \square  \  F}_{n -\mbox{times}}, \mbox{ and } F^0 = (\{1\}, \emptyset).
\end{equation}

Inspired by the usual definition of topological $K$-theory of a topological space in \cite[Chapter I, \S 11]{ttd}, denote by $[X]$ be the equivalence class of the $F^n$-graph bundle $(X, p, \Gamma)$. From Lemma \ref{stab}, we are allowed to define without ambiguities an operation between equivalence classes in the following way
\begin{equation}\label{classoperation}
    [X] \boxplus [Y] := [X \boxplus Y]. 
\end{equation}
From Theorem \ref{5thmaintheorem}, the set of the equivalence classes $[X]$ with respect to \eqref{classoperation} is an abelian monoid. On an abelian monoid $(A, +)$, its \textit{Grothendieck group} $(G(A),+)$, or its \textit{group of differences}, or its \textit{group obtained by symmetrization}, is the abelian group defined on  $A$ having binary operation on equivalence classes in $A \times A$ such that $(a,b) \sim (c,d)$ if and only if there is an element $r$ in $A$ such that $a + c + r = b + d + r$. Then $G(A)$ is the set of the equivalence classes with respect to the sum and the class of the element $(a,b)$ is denoted by $a - b$. The operation is defined componentwise, i.e.:     $(a-b)+(c-d) = (a+c) - (b+d).$ The neutral element of this group is the class whose element are $(a,a)$. Using the notation of differences   $a-a = 0_{G(A)}$ for each $a$ in $A$. Moreover the inverse of the element $a-b$ is the element $b-a$ for each couple $a,b$ in $A$.

For instance, the Grothendieck group of the usual additive monoid of the natural numbers is given by the additive group of the integers, but is more interesting to note that given a topological space $M$ the Grothendieck group of the monoid of the equivalence classes of vector bundles over $M$ with respect to the direct sum is  the well known \textit{K-theory group} of $M$ (denoted by $K_0(M)$).

The role of $K_0(M)$ in algebraic topology is fundamental under several aspects, in order to classify manifolds, surfaces, polytopes and CW-complexes, see \cite{ttd}. It seems quite natural to define a $K$-theory group for a graph induced by its graph bundles.


\begin{defn}\label{nktg}
Let $\Gamma$ and $F$ be two graphs. The \textbf{K-theory group of the network} $\Gamma$ (or the \textbf{network K-theory group} of $\Gamma$)  is the group $K_0(\Gamma, F)$ defined as the Grothendieck group of the monoid of the equivalence classes of $F^n$-graph bundles over $\Gamma$ and $\boxplus$ as operation.
\end{defn}

\begin{example}
    Let $\Gamma = T$ be a tree. Then, for any $F$, its group $K_0(\Gamma, F) = \mathbb{Z}$. Indeed any graph bundle on a tree is trivial (see Remark \ref{remtree}). This means that the monoid of the classes of equivalence of vector bundle is just $\mathbb{N}$: the isomorphism is the one which sends each $F^n$-graph bundle to $n$. Therefore the Grothendieck group is $G(\mathbb{N}) = \mathbb{Z}$.
\end{example}

As it often happens in algebraic topology when we have similar situations, we look whether $K_0(\cdot, F)$ is a controvariant functor or not. In this case the answer is positive (of course from the category of graphs endowed with graph morphisms to the category of the abelian groups endowed with group homomorphisms). 

\begin{cor}\label{k.t.func}
Given the graphs $\Gamma_1,$ $ \Gamma_2$  and $F$ and the network $K$-theory groups $K_0(\Gamma_1, F)$ and $ K_0(\Gamma_2, F)$ as per Definition \ref{nktg}, for each graph morphism $f: \Gamma_1 \longrightarrow \Gamma_2$ there is a group morphism $K_0(f): K_0(\Gamma_2, F) \longrightarrow K_0(\Gamma_1, F)$ such that the following functorial properties are satisfied:
\begin{itemize}
    \item[{\rm (1).}] $K_0(f \circ g) = K_0(g) \circ K_0(f)$,
    \item[{\rm (2).}] $K_0(\mathrm{id}_\Gamma) = \mathrm{id}_{K_0(\Gamma, F)}.$
\end{itemize}
\end{cor}
\begin{proof}
The group morphism $K_0(f)$ is defined as $K_0(f)([X]-[Y]) = [f^*X] - [f^*Y]$.
The properties follow by  Proposition \ref{funct}, Theorem \ref{sumstab} and Corollary \ref{pulltrivial} with arguments of routine concerning the notion of Grothendieck group. 
\end{proof}

We end by stating an open problem which we encountered in our investigations.

\begin{problem}
    The computation of the network $K$-theory group $K_0(\Gamma, F)$ of a graph $\Gamma$ may appear quite challenging,  especially when the automorphism group of the graph $F$ is particularly wide. On the other hand, it could be a useful tool for classifying graphs. Furthermore, there seem to be no major difficulties in defining the concept of reduced $K$-theory of a network: this, similarly to the classical case, should be easier to compute and could still retain the information contained in the $K$-theory group.
\end{problem}

\section{The Cayley graph of a subdirect product of groups}\label{section5}
In this section, we will study the Cayley graphs of subdirect products of finite groups. These are graphs constructed from a group and a set of its generators. Cayley graphs are one of the most widely used tools in graph theory for the study of groups.

\begin{rem}
Note from \cite{bddr, spectrabook, crs, gross2} that for a finite group $G$ with $S$ set of generators such that $1_G \notin S$, we say that $S$ is  \textit{symmetric} if for every $s$ in $S$, then $s^{-1}$ is in $S$. Then the \textit{Cayley graph} of $G$ with respect to $S$ is the graph $Cay(G,S)$ whose vertex set is $G$ and $x \sim y$ if and only if there is a an element $s$ in $S$ such that $xs = y$. Note also that the $Cay(G,S)$ is a well defined graph since $S$ is symmetric. For a set $S$ which is not symmetric, then there is a notion of $Cay(G,S)$ in terms of  \textit{directed graph}. It is important to notice that for each group $G$, it is always possible to find a symmetric generators of $S$, and so it is possible to define $Cay(G, S)$. 
\end{rem}

The Cayley graphs defined with different choices of $S$ can be not isomorphic. This fact is highlighted in Fig. \ref{Z4cart} and (coherently with the notations in \cite{bbr1, b, bddr, hofmor}) we denote the cyclic group  on $n$ elements by $\mathbb{Z}(n)$ (here $n$ is a given natural number).

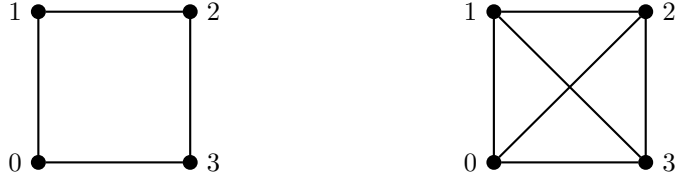
\begin{figure}[h]
\begin{tikzpicture}[scale=1 ]
\node[vertex,fill=black] (11) at  (-4, 1) [scale=0.8,label=left:$1$] {};
\node[vertex,fill=black] (12) at  (-2, 1) [scale=0.8,label=right:$2$] {};
\node[vertex,fill=black] (13) at  (-2, -1) [scale=0.8,label=right:$3$] {};
\node[vertex,fill=black] (14) at  (-4, -1) [scale=0.8,label=left:$0$] {};

\node[vertex,fill=black] (31) at  (2, 1) [scale=0.8,label=left:$1$] {};
\node[vertex,fill=black] (32) at  (4, 1) [scale=0.8,label=right:$2$] {};
\node[vertex,fill=black] (33) at  (4, -1) [scale=0.8,label=right:$3$] {};
\node[vertex,fill=black] (34) at  (2, -1) [scale=0.8,label=left:$0$] {};

\draw[edge, thick]  (11) --  (12) node[midway, below left] {};
\draw[edge, thick]  (12) --  (13) node[midway, below left] {};
\draw[edge, thick]  (13) --  (14) node[midway, below left] {};
\draw[edge, thick]  (14) --  (11) node[midway, below left] {};

\draw[edge, thick]  (31) --  (32) node[midway, below left] {};
\draw[edge, thick]  (32) --  (33) node[midway, below left] {};
\draw[edge, thick]  (33) --  (34) node[midway, below left] {};
\draw[edge, thick]  (34) --  (31) node[midway, below left] {};
\draw[edge, thick]  (31) --  (33) node[midway, below left] {};
\draw[edge, thick]  (32) --  (34) node[midway, below left] {};
\end{tikzpicture}\caption{on the left, $Cay(\mathbb{Z}(4), \{1,3\})$; on the right, $Cay(\mathbb{Z}(4), \{1,2,3\})$ }\label{Z4cart}
\end{figure}

This observation motivates the search for different generating sets. Indeed, the properties of Cayley graphs corresponding to different generating sets could reflect distinct properties of the group. Moreover, not all choices of generating sets may be truly useful (consider, for example, that if $S = G \setminus \{1_G\}$, then the Cayley graph is the complete graph regardless of $G$). Therefore, the choice of $S$ becomes particularly significant in making the Cayley graph an effective tool. For an example of this fact see \cite{bddr}.

\subsection{Subdirect product of groups}
In this subsection, we will introduce the concept of a subdirect product of groups. To this end, it is necessary to recall the following proposition.

\begin{prop}[See \cite{dh}, Chapter A, Theorem 19.1] \label{subdirectproducttheorem}
Let $A, B$ and $C$ be finite groups along with $\varepsilon_A: A \to C$ and $\varepsilon_B: B \to C$ epimorphisms with $\Delta_A = \ker (\varepsilon_A)$ and $\Delta_B = \ker (\varepsilon_B)$. If we consider $A \times B$,  $\bar{A}= A \times 1$, $\bar{B}=1 \times B$ and \begin{equation}E = \{(a,b) \mid a \in A, b \in B, \text{ and } \varepsilon_A(a)=\varepsilon_B(b) \}\end{equation}
then $E$ is a subgroup of $A \times B$ and there exist epimorphisms $\delta_A : E \to A$ and $\delta_B : E \to B$  such that the following conditions are satisfied:
\begin{itemize}
\item[{\rm (1).}] $\ker (\delta_A)=E \cap \bar{B} \cong \Delta_B$ and $\ker (\delta_B)=E \cap \bar{A} \cong \Delta_A$,
\item[{\rm (2).}] $\ker (\delta_A) \ker(\delta_B)=\Delta_A \times \Delta_B$, and
\item[{\rm (3).}] $E/(\Delta_A \times \Delta_B) \cong C$.
\end{itemize}
\end{prop}

Proposition \ref{subdirectproducttheorem} may be visualized in the following way: if we have surjective homomorphisms of groups \begin{equation} \varepsilon_A \ : \   a \in A \longmapsto \varepsilon_A(a) \in C  \ \ \mbox{and} \ \  \varepsilon_B \ : \  b \in B \longmapsto \varepsilon_B(b) \in C,\end{equation} then the canonical projections \begin{equation}\mathrm{pr}_A  :  (a,b) \in A \times B \mapsto \mathrm{pr}_A(a,b)=a \in A, \    \mathrm{pr}_B  :  (a,b) \in A \times B \mapsto \mathrm{pr}_A(a,b)=b \in B \end{equation} allow us to have the following commutative diagram
\begin{equation}
\adjustbox{scale=1.5}{
\begin{tikzcd}
E
\arrow[twoheadrightarrow, rrd, bend left, dashed, "\delta_B"]
\arrow[twoheadrightarrow, ddr, bend right, dashed, "\delta_A"]
\arrow[hook, rd, "\iota"] \\
& A \times B \arrow[twoheadrightarrow, r, "\mathrm{pr}_B"] \arrow[twoheadrightarrow, d, "\mathrm{pr}_A"]
& B \arrow[twoheadrightarrow, d, "\varepsilon_B"] \\
& A \arrow[twoheadrightarrow, r, "\varepsilon_A"]
& C \end{tikzcd}}
\end{equation}
where $E$ embeds naturally in $A \times B$ via $\iota$. This commutative diagram shows Proposition \ref{subdirectproducttheorem} (3), i.e.: 
\begin{equation}\label{subdirectproduct}
C \simeq \frac{E}{\Delta_A   \times  \Delta_B}, \ \ \mbox{where} \ \  \Delta_A \simeq \ker (\delta_A)  \ \mbox{and} \ \Delta_B \simeq  \ker (\delta_B).  
\end{equation}

\begin{defn}[See \cite{dh}, Chapter A, Definition 19.2]\label{subdirgroups}
We say that the subgroup $E$ of $A \times B$ from Proposition \ref{subdirectproducttheorem} is a \textit{subdirect product of (groups) $A$ and $B$ with amalgamated factor group $C$}, writing briefly $A \times_C B$. 
\end{defn}
\begin{rem}
Given three groups $A$, $B$, and $C$, and two morphisms $\phi_1: A \longrightarrow C$ and $\phi_2: B \longrightarrow C$, it is possible to define the \textit{pullback} of $B$ along $\phi_1$ by setting  
\begin{equation}
\phi_1^*B = \{(a,b) \in A \times B \mid \phi_1(a) = \phi_2(b)\}
\end{equation}  
endowed with the componentwise product.  
It can be observed that the subdirect product of two groups is nothing but the pullback of a group in the case where all the morphisms involved are surjective. This is the same approach we took in our manuscript when defining the subdirect product of a graph bundle: indeed, we first defined the notion of the pullback of a graph bundle along a generic graph morphism, and then we called \textit{subdirect product} the pullback obtained through a surjection that, in turn, arises from another graph bundle.
\end{rem}
\subsection{Graph bundles and morphisms of groups}
In this section, we aim to discuss the Cayley graphs of subdirect products of groups. As previously illustrated, the choice of generators is crucial for understanding the properties of the group. In particular, we will consider generating sets described by the following definition. Recall that every generating set we consider does not contain the identity of the group.

\begin{defn}\label{crosssectionadapted}
Let $\varphi: A \longrightarrow B$ be a surjective morphism between to finite groups $A$ and $B$. We say that a symmetric system of generators $S_1=\{s_1, \dots, s_n\}$   for $B$ is an $S_1$\textbf{-transversal section of $\varphi$}, if it is a symmetric subset $\widetilde{S}_1=\{\widetilde{s}_1, \dots, \widetilde{s}_n\}$ of $A$ such that $\varphi(\widetilde{s}_i) = s_i$ for each $i=1, \dots, n$.  We say that a symmetric system of generators $S_0$ for $\ker(\varphi)$ is \textbf{admissible} if
\begin{equation}
a \ S_0 \ a^{-1} \subseteq S_0 \ \ \ \ \  \mbox{if for each}  \ a \in A.
\end{equation}
We call \textbf{system of generators induced by  $\varphi$ and $S_1$}  a symmetric system of generators of $A$ which is defined by
\begin{equation}
    S_{\varphi} = S_0 \cup \widetilde{S}_1,
\end{equation}
where $\widetilde{S}_1$ is a $S_1$-transversal section and $S_0$ is admissible.
\end{defn}

Two observations are immediate here.

The first  is that for a given  surjective morphism of groups $\varphi : A \longrightarrow B$ and a symmetric system of generators $S_1$, it is always possible to construct a symmetric system of generators induced by $\varphi$ and $S_1$. Indeed, since $\varphi$ is surjective, it is always possible to find a $S_1$-transversal section of $\varphi$. Moreover, if necessary, $S_0 = \ker(\varphi)$ is admissible. In other words, there are always natural examples for Defintion \ref{crosssectionadapted}.

The second observation is that  whenever we have $s_i$ and $s_j$  generators in $S_1$ such that $s_j = s_i^{-1}$, if we choose an $S_1$-trasversal section $\widetilde{S}_1$ of $\phi$, then $\widetilde{s}_j^{-1} = \widetilde{s}_i^{-1}$, since $\widetilde{S}_1$ is symmetric.

We are ready to illustrate one of the most important results of the present contribution. Roughly speaking, under an appropriate choice of system of generators, we prove that the Cayley graph of the domain of a surjective morphism of groups is the total space of a graph bundle. In particular the base  is the Cayley graph of the image and the fiber is the Cayley graph of the kernel  of the morphism. An easy example is the case of a projection form a Cartesian product of groups onto one of its components. However our result is far more general and  Example \ref{last} shows some relevant circumstances in which our result applies.

\begin{thm}[Total Spaces of Graph Bundles as Cayley Graphs]\label{6thmaintheorem}
Let $\varphi: A \longrightarrow B$ be a surjective morphism between two finite groups $A$ and $B$ and let $S_1$ be a symmetric system of generators for $B$. Denote by $\widetilde{S}_1$ a transversal section of $B$, by $S_0$ a system of generators which is admissible and by $S_\varphi = S_0 \cup \widetilde{S}_1$ the system of generators of $A$ induced by $S_1$ and $\varphi$. Then the triple $(Cay(A, S_{\varphi}), \varphi, Cay(B, S_1))$ is a $Cay(\ker(\varphi), S_0)$-bundle.
\end{thm}

\begin{proof}
First of all, note that $\varphi: Cay(A, S_{\varphi}) \longrightarrow Cay(B, S_1)$ is a surjective morphism of graphs. Indeed $\varphi$ is surjective by assumptions and it is a morphism of graphs since $\varphi(x) = \varphi(xh_i)$ for each $h_i$ in $S_0$ and $\varphi(x) \sim \varphi(x\widetilde{s}_i) = \varphi(x)s_i$ for each $\widetilde{s}_i$ in $\widetilde{S}_1$.
In order to conclude the proof it is sufficient to check  Definition \ref{graphbundle}:
\begin{itemize}
    \item[Step 1.] Given $b$ in $B$, denote by $\varphi^{-1}(b)$ the subgraph of $Cay(A, S_{\varphi})$ induced by the elements in the preimage of $b$. Let $x$  be in the preimage of $v$ and consider
\begin{equation}
       \sigma_b \ : \ \varphi^{-1}(b) \longrightarrow Cay(\ker(\varphi), S_0) \  \ \mbox{such that} \ \ y \longmapsto yx^{-1}. 
\end{equation}
The map $\sigma_b$ is an isomorphism of graphs. Indeed for each $y$ in the preimage of $b$, we have that $yx^{-1}$ is an element of $\ker(\varphi)$. Moreover $\sigma_b$ is bijection. It is necessary to prove that $\sigma_b$ and its inverse preserve the edges. Let $y$ and $z$ in the preimage of $b$. They are connected in $\varphi^{-1}(b)$ if and only if there is a $h_0$ in $S_0$ such that $yh_0 = z$. Then observe that 
\begin{equation}
    \sigma_b(z) = zx^{-1} = yh_0x^{-1} = yx^{-1}xh_0x^{-1} = \sigma_b(y)xh_0x^{-1} = \sigma_b(y)h_1,
\end{equation}
where $h_1$ is an element of $S_0$, since $S_0$ is admissible. This implies that $\sigma_b(z) \sim \sigma_b(y)$. On the other hand, let $\sigma_b^{-1}(k) = kx$ and let $k \sim kh_0$ in $Cay(\ker(\varphi), S_0)$. Then \begin{equation}\sigma_b^{-1}(kh_0) = \sigma_b^{-1}(k) x^{-1}h_0x = \sigma_b^{-1}(k)h_2,\end{equation} where $h_2$ is in $S_0$, because $S_0$ is admissible. Then also $\sigma_b^{-1}$ preserves the edges.

\item[Step 2.] Denote by $\widetilde{Cay(A, S_\phi)}$ the graph with the same vertex set of $Cay(A, S_\phi)$ but whose edges are the ones induced by $\widetilde{S}_1$. Let $\widetilde{\varphi}:\widetilde{Cay(A, S_\phi)} \longrightarrow Cay(B, S_1)$ be the graph morphism induced by $\varphi$. Then $\widetilde{\varphi}$ is a covering.
In order to prove this fact, given $b$ in $B$ and $x$ in the preimage of $b$, it is necessary to show that the restriction of $\varphi$ to the neighborhoods $N(x)$ and $N(b)$ of $x$ and $b$ respectively is an isomorphism of graphs. This follows easily from 
\begin{equation}
    \varphi(x\widetilde{s}_i) = \varphi(x)\varphi(\widetilde{s}_i) = bs_i,
\end{equation}
and from the fact that the map which associate to each $s_i$ the element $\widetilde{s}_i$ is a bijection.
Moreover, this implies the existence for each $b \sim c$ in $B$ of a bijection $\psi_{b,c}: \varphi^{-1}(b) \longrightarrow \varphi^{-1}(c)$. In particular, if $c = bs_i$, then
\begin{equation}
\psi_{b,bs_i}(x) = x\widetilde{s}_i. 
\end{equation}
\item[Step 3.] Now $\psi_{b,bs_i}: \varphi^{-1}(b) \longrightarrow \varphi^{-1}(bs_i)$ is an isomorphism of graphs. Consider $x \sim xh_0$ in $\varphi^{-1}(b)$. Then 
\begin{equation}
    \psi_{b,bs_i}(x) = x\widetilde{s}_i  \ \mbox{  and  }  \ \psi_{b,bs_i}(xh_0) = xh_0\widetilde{s}_i = \psi_{b,bs_i}(x)\widetilde{s}_i^{-1}h_0\widetilde{s}_i.
\end{equation}
Since $S_0$ is admissible,  $\widetilde{s}_i^{-1}h_0\widetilde{s}_i$ is in $S_0$ and so $\psi_{b,bs_i}$ preserves the edges. Arguing in the  same way, also $\psi_{b,bs_i}^{-1} = \psi_{bs_i,b}$ preserves the edges and this implies that $\psi_{b,bs_i}$ is an isomorphism.
\end{itemize}
\end{proof}

It is instructive to offer examples, not only as evidence of Theorem \ref{6thmaintheorem}, but to observe that the graphs of Example \ref{exem.bundle} (and other constructions which we have seen until now) appear also as  Cayley graphs.

\begin{example}\label{exem.uno}
Let $\phi_1: \mathbb{Z}(2) \times \mathbb{Z}(3) \longrightarrow \mathbb{Z}(3)$ be homomorphism of group  $\phi_1(x,y)= y$ which projects onto the second factor of the direct product $\mathbb{Z}(2) \times \mathbb{Z}(3)$. Here $\ker(\phi_1) = \mathbb{Z}(2) \times \{0\}$. Then consider $S_{0} = \{(1,0)\}$ as generating set for $\ker(\phi_1)$ and let $S_1 = \{1\}$ be a generating set for $\phi_1(\mathbb{Z}_2 \times \mathbb{Z}(3)) = \mathbb{Z}(3)$. Define the $S_1$-transversal section $\{(0,1)\}$ and let $S_{\phi_1} = \{(1,0), (0,1)\}$. Here the graph bundle $(Cay(\mathbb{Z}(2) \times \mathbb{Z}(3), S_{\varphi}), \phi_1, Cay(\mathbb{Z}(3), S_1))$ is exactly the 2-fold graph bundle $(K_2  \ \square \  C_3, p, C_3)$ in Example \ref{exem.bundle}.
\end{example}

With the same approach, we find another construction which we have discussed before.

\begin{example}\label{exem.due}
Now we consider $\phi_2: \mathbb{Z}(6) \longrightarrow \mathbb{Z}(3)$ homorphism groups which is defined as $\phi_2(x)= [x]$ where $[x]$ is the remainder of the division of $x$ with $3$. Observe that $\ker(\phi_2) = \{0,3\} \cong \mathbb{Z}(2)$. Fix $S_0=\{3\}$, as the generating set for $\ker(\phi_2)$, and choose $S_1=\{1\}$ so that $\widetilde{S}_1 = \{1\}$. Finally define $S_{\phi_2} = \{0,1,3\}$. Then the graph bundle $(Cay(\mathbb{Z}(6), S_{\varphi}), \phi_1, Cay(\mathbb{Z}(3), S_1))$ is the $K_2$-graph bundle $(M_3, q, C_3)$ which we have always seen in Example \ref{exem.bundle}.
\end{example}

Now we consider Definition \ref{subdirgroups}, beginning with $\phi_1: A_1 \longrightarrow B$ and $\phi_2: A_2 \longrightarrow B$  surjective homomorphisms of finite groups. We then form the subdirect product of groups $A_1 \times_{B} A_2$ with amalgamating subgroup $B$ and have by default $\varphi: A_1 \times_{B} A_2 \longrightarrow B$  surjective homomorphism of groups. Let $S_1$ be a symmetric system of generators of $B$, choose an $S_1$-section $\widetilde{S}_{1,i} = \{\widetilde{s}_{1,i}, \dots, \widetilde{s}_{n,i}\}$ of $\phi_i$ for each $i=1,2$ and  $S_{0,i}$ a symmetric system of generators of $\ker(\phi_i)$ which is admissible. Finally, fix the system of generators induced by $\phi_i$ and $S_1$ given by $S_{\phi_i} = \widetilde{S}_{1,i} \cup S_{0,i}$. On the other hand, we may define the $S_1$-section of $\varphi$
\begin{equation}\label{instrumental1}
\widetilde{S}_1 = \{(\widetilde{s}_{1,1},\widetilde{s}_{1,2}), \dots, (\widetilde{s}_{n,1},\widetilde{s}_{n,2})\}.
\end{equation}
Moreover let $S_0 = \overline{S}_{0,1} \cup \overline{S}_{0,2}$, where 
\begin{equation}\label{instrumental2}
\overline{S}_{0,1} = S_{0,1} \times {1_{A_2}} \mbox{ and } \overline{S}_{0,2} = {1_{A_1}} \times S_{0,2}. 
\end{equation}
Therefore we may consider also \begin{equation}\label{instrumental3}
S_{\varphi}=S_0 \cup \widetilde{S}_1
\end{equation}  system of generators induced by $\varphi$ and $S_1$.

We are in the position to prove a significant generalization of \cite[Theorem 3.6]{bddr}, which is somehow the main structural result behind several ideas and constructions  in \cite{bbr1,  b, bddr}.

\begin{thm}[Invariance of Subdirect Products by Cayley Operator]\label{7thmaintheorem}
Let $\phi_1: A_1 \longrightarrow B$ and $\phi_2: A_2 \longrightarrow B$ be two surjective homomorphisms of groups and consider the subdirect product of groups $ A_1 \times_{B} A_2$ along with \eqref{instrumental1}, \eqref{instrumental2} and \eqref{instrumental3}. Then 
\begin{equation}
Cay(A_1 \times_{B} A_2, S_{\varphi}) = Cay(A_1,S_{\phi_1}) \ \boxplus  \ Cay(A_2,S_{\phi_2}).
\end{equation}
\end{thm}
\begin{proof}
First, we note that $Cay(A_1 \times_{B} A_2, S_{\varphi})$ and $Cay(A_1,S_{\phi_1}) \boxplus Cay(A_2,S_{\phi_2})$ have the same set of vertices, which is $A_1 \times_{B} A_2$.

Let $\{(x,y),(x',y')\}$ be an edge of $Cay(A_1 \times_{B} A_2, S_{\varphi})$. Assume that there is an element $(h_1, 1_{A_2})$ in $\overline{S}_{0,1}$ such that $(x', y') = (x,y)(h_1,1_{A_2})=  (xh_1,y)$. This is an edge of $Cay(A_1,S_{\phi_1}) \boxplus Cay(A_2,S_{\phi_2})$, in particular it is an edge of type II, if we consider the subdirect product as the total space of the graph bundle $\phi_1^*(Cay(A_2,S_{\phi_2}), \phi_2, Cay(B,S_1))$. On the other hand,  $(1_{A_1}, h_2)$ is an element of $\overline{S}_{0,2}$ such that $(x', y') = (x,y)(1_{A_1}, h_2) =  (x,yh_2)$. Then the edge $\{(x,y), (x,yh_2)\}$ can be seen as an edge of type I of the graph bundle $\phi_1^*(Cay(A_2,S_{\phi_2}), \phi_2, Cay(B,S_1))$. 
Finally, let $(\widetilde{s}_{j,1},\widetilde{s}_{j,2})$ be an element in $\widetilde{S}_1$ such that $(x', y') = (x,y)(\widetilde{s}_{j,1},\widetilde{s}_{j,2}) =  (x\widetilde{s}_{j,1},y\widetilde{s}_{j,2})$. Then the edge $\{(x,y),(x', y')\}$ is a type III edge of the graph bundle $\phi_1^*(Cay(A_2,S_{\phi_2}), \phi_2, Cay(B,S_1))$. The result follows.
\end{proof}

\begin{figure}
\begin{tikzpicture}[scale=1 ]
\node[vertex,fill=black] (11) at  (-5, 1) [scale=0.8,label=left:$1a$] {};
\node[vertex,fill=black] (12) at  (-3, 1) [scale=0.8,label=right:$1b$] {};
\node[vertex,fill=black] (13) at  (-3, -1) [scale=0.8,label=below right:$1c$] {};
\node[vertex,fill=black] (14) at  (-5, -1) [scale=0.8,label=left:$1d$] {};

\node[vertex,fill=black] (21) at  (-1, 6.75) [scale=0.8,label=left:$2a$] {};
\node[vertex,fill=black] (22) at  (1, 6.75) [scale=0.8,label=right:$2b$] {};
\node[vertex,fill=black] (23) at  (1, 4.75) [scale=0.8,label=right:$2c$] {};
\node[vertex,fill=black] (24) at  (-1, 4.75) [scale=0.8,label=left:$2d$] {};

\node[vertex,fill=black] (31) at  (3, 1) [scale=0.8,label=left:$3a$] {};
\node[vertex,fill=black] (32) at  (5, 1) [scale=0.8,label=right:$3b$] {};
\node[vertex,fill=black] (33) at  (5, -1) [scale=0.8,label=right:$3c$] {};
\node[vertex,fill=black] (34) at  (3, -1) [scale=0.8,label=below left:$3d$] {};

\draw[edge, thick]  (11) --  (12) node[midway, below left] {};
\draw[edge, thick]  (12) --  (13) node[midway, below left] {};
\draw[edge, thick]  (13) --  (14) node[midway, below left] {};
\draw[edge, thick]  (14) --  (11) node[midway, below left] {};

\draw[edge, thick]  (21) --  (22) node[midway, below left] {};
\draw[edge, thick]  (22) --  (23) node[midway, below left] {};
\draw[edge, thick]  (23) --  (24) node[midway, below left] {};
\draw[edge, thick]  (24) --  (21) node[midway, below left] {};

\draw[edge,thick]  (31) --  (32) node[midway, below left] {};
\draw[edge,thick]  (32) --  (33) node[midway, below left] {};
\draw[edge,thick]  (33) --  (34) node[midway, below left] {};
\draw[edge,thick]  (34) --  (31) node[midway, below left] {};

\draw[edge]  (11) --  (21) node[midway, below left] {};
\draw[edge]  (12) --  (22) node[midway, below left] {};
\draw[edge]  (13) --  (23) node[midway, below left] {};
\draw[edge]  (14) --  (24) node[midway, below left] {};

\draw[edge]  (31) --  (21) node[midway, below left] {};
\draw[edge]  (32) --  (22) node[midway, below left] {};
\draw[edge]  (33) --  (23) node[midway, below left] {};
\draw[edge]  (34) --  (24) node[midway, below left] {};

\draw[edge]  (11) --  (34) node[midway, below left] {};
\draw[edge]  (12) --  (33) node[midway, below left] {};
\draw[edge]  (13) --  (32) node[midway, below left] {};
\draw[edge]  (14) --  (31) node[midway, below left] {};

\end{tikzpicture}\caption{The graph $(K_2 \square C_3) \boxplus M_3$.}\label{final}
\end{figure}
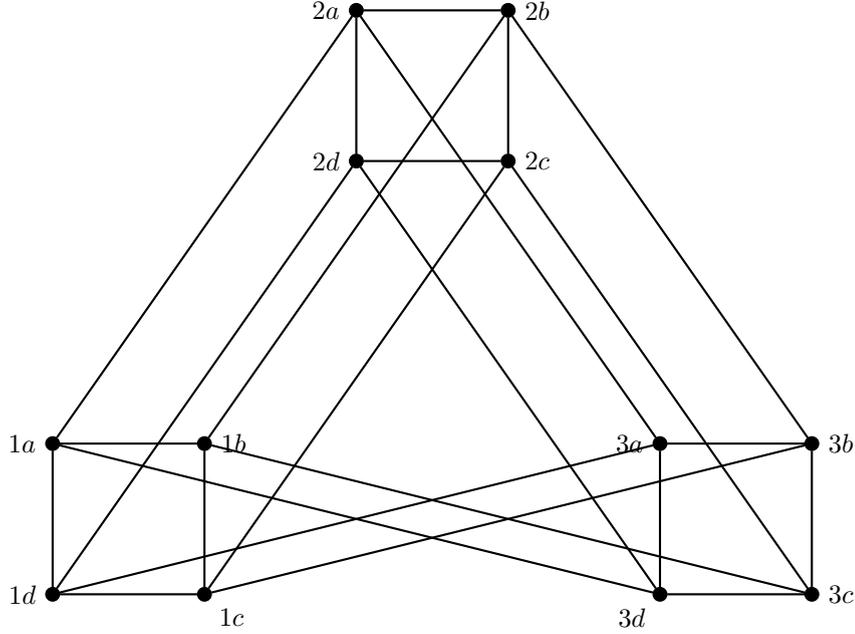

It is instructive to refer to Theorem \ref{7thmaintheorem} as a "result of invariance", since it allows us to look at "$Cay$" as an operator from finite groups to finite graphs with the nice property of transformiing a subdirect product of finite groups in a subdirect product of graph bundles. 

\begin{example}\label{last}
 Let $\phi_1: \mathbb{Z}(3) \times \mathbb{Z}(3) \longrightarrow \mathbb{Z}(3)$ and $\phi_2: \mathbb{Z}(6) \longrightarrow \mathbb{Z}_3$ be the homomorphisms of groups in Examples \ref{exem.uno} and \ref{exem.due}. We now consider the subdirect product of these abelian groups, namely \begin{equation}
 \mathcal{Z} = (\mathbb{Z}(3) \times \mathbb{Z}(3))\times_{\mathbb{Z}(3)} \mathbb{Z}(6)\end{equation}  and denote by $\varphi: \mathcal{Z} \longrightarrow  \mathbb{Z}(3)$  the homomorphism of groups which is induced by $(\phi_1, \phi_2)$.
 
 In this situation, if  $\overline{S}_{0,1} = \{(1,0,0)\}$ and $\overline{S}_{0,2} = \{(0,0,3)\}$. Then $S_0= \{(1,0,0), (0,0,3)\}$ is the admissible generating set of $\ker(\varphi)$. On the other hand, if $S_1 = \{1\}$ and  $\widetilde{S}_1 = \{(0,1,1)\}$,  then  the $Cay(\ker(\varphi), S_0)$-graph bundle $(Cay(\mathcal{Z}, S_{\varphi}), \varphi, Cay(\mathbb{Z}(3), S_1))$ is exactly the graph bundle $((K_2 \square C_3), p, C_3) \boxplus (M_3, q, C_3)$, whose total space is depicted in Fig. \ref{final}.
\end{example}

We end with an open problem: it would be interesting to study the spectrum of the total spaces of these graph bundles and, in particular, their applications on the structural properties of relevan classes of finite groups. It does not seem possible to find a precise and general description of the spectrum, exactly as for graph coverings, which are a special case of graph bundles. However, it might be possible to find some spectral results for some specific families of graphs.

\begin{problem} Find  spectral properties of the graph bundles which can be constructed as subdirect product of Cayley graphs as per Theorem \ref{7thmaintheorem}.    
\end{problem}


\begin{thebibliography}{99}

\bibitem{alon}N.  Alon,  A. Lubotzky and A. Wigderson: Semi-direct product in groups and zig-zag product in graphs: connections and applications (extended abstract). In: 42-nd IEEE Symposium on Foundations of Computer Science, Las Vegas, NV, 2001, pp. 630–637. IEEE Computer Society, Los Alamitos, CA (2001).



\bibitem{bbr1} F. Bagarello, Y. Bavuma and F.G. Russo, Topological decompositions of the Pauli group and their influence on dynamical systems, \textit{Math. Phys. Anal. Geom.} \textbf{24} (2021), Article No. 16.dc



\bibitem{banic} I. Banič and J. Žerovnik, Wide diameter of Cartesian graph bundles, \textit{Discrete Math.} \textbf{310}  (2010), 1697--1701. 

\bibitem{b} Y. Bavuma, A short note on the topological decomposition of the central product of groups, \textit{Trans. Comb.} \textbf{11} (2022), 123--129.

\bibitem{bddr} Y. Bavuma, D. D’Angeli, A. Donno and F.G. Russo, On an infinite family of integral Cayley graphs of Pauli groups, \textit{J. Algebra} \textbf{659} (2024), 148--182.

\bibitem{belardo1} F. Belardo, M. Brunetti and S. Khan, NEPS of complex unit gain graphs, \textit{Electron. J. Linear Algebra} \textbf{39} (2023), 621--643.


\bibitem{spectrabook} A.E. Brouwer and W.H. Haemers, \textit{Spectra of graphs}, Springer, New York, 2012.


\bibitem{godsil} M. Cavaleri, A. Donno and S. Spessato, Godsil-McKay switchings for gain graphs, \textit{Electron. J. Linear Algebra} \textbf{41} (2025), 21--48.





\bibitem{clarke}F.W. Clarke, A.D. Thomas and D.A. Waller, Embeddings of covering projections of graphs, \textit{J. Comb. Theory Ser. B} \textbf{28} (1980),  10--17.



\bibitem{neps1} D.M. Cvetkovi\'{c} and R.P. Lu\v{c}i\'{c}, A new generalization of the concept of the $p$-sum of graphs, \emph{Univ. Beograd. Publ. Elektrotehn. Fak. Ser. Mat. Fiz.} {\bf 302} (1970),  67--71.

\bibitem{crs} D. Cvetkovi\'{c}, P. Rowlinson and S. Simi\'{c}, \textit{An introduction to the theory of graph spectra}, Cambridge University Press, Cambridge, 2010. 

\bibitem{neps2} D. Cvetkovi\'{c} and S. Simi\'{c}, Non-complete extended $P$-sum of graphs, graph angles and star partitions, \emph{Publ. Inst. Math. (Beograd) (N.S.)} {\bf 53}  (1993), 4--16.


\bibitem{ttd} T. tom Dieck, \textit{Transformation groups},  de Gruyter,  Berlin, 1987.


\bibitem{dh} K. Doerk and T. Hawkes, \textit{Finite soluble groups},  de Gruyter,  Berlin, 1994.

\bibitem{donno1} A. Donno, Generalized wreath products of graphs and groups, \textit{Graphs Comb.} \textbf{31} (2015), 915--926.



\bibitem{dragan1} A. W. M. Dress and D. Stevanovi\'c, A note on a theorem of Horst Sachs,
\textit{Ann. Comb.} \textbf{8} (2004), 487--497.  



\bibitem{gross2} L. Gross and T. W. Tucker, \textit{Topological graph theory},   Wiley and Sons, New York, 1987.



\bibitem{handbook} R. Hammack, W. Imrich and S. Klavžar, \textit{Handbook of product graphs},  CRC Press, Boca Raton, 2011.





\bibitem{hofmor} K.H. Hofmann and S.Morris, \textit{The structure of compact groups}, de Gruyter, Berlin, 2006.







\bibitem{kos} C. Kosniowski, \textit{Introduction to algebraic topology}, Cambridge University Press,  Cambridge, 1980.


\bibitem{kwak} J.H. Kwak and J. Lee, Isomorphism classes of graph bundles, \textit{Canadian J.  Math.} \textbf{42} (1990), 747--761.

\bibitem{kwakk} J.H. Kwak and Y.S. Kwon, Characteristic polynomials of graph bundles having voltages in a dihedral group, \textit{Linear Algebra Appl.}  \textbf{336}  (2001),  99--118.

\bibitem{tensore} I. H. Ladinek and J. Žerovnik, On connectedness and hamiltonicity of direct graph bundles,  \textit{Math. Commun.} \textbf{17} (2012), 21-–34.


\bibitem{larrion} F. Larri\'on, M.A. Piza\~na and R. Villarroel-Flores, On strong graph bundles, \textit{Discrete Math.} \textbf{340} (2017),  3073--3080.


\bibitem{lessico} B. Mohar, T. Pisanski and M. Škoviera. The maximum genus of graph bundles,  \textit{European J. Combin.} \textbf{9} (1988),  215-–224.



\bibitem{Pisanski} T. Pisanski, J. Shawe-Taylor and  J. Vrabec, Edge-colorability of graph bundles, \textit{J. Comb. Theory Ser. B} \textbf{35} (1983), 12--19.

\bibitem{ps1} M. Planat and M. Saniga, On the Pauli graphs on $N$-qudits, \emph{Quantum Inf. Comput.} {\bf 8} (2008),  127--146.


\bibitem{ps2} M. Planat and M. Saniga, A sequence of qubit-qudit Pauli groups as a nested structure of doilies, \emph{J. Phys. A} {\bf 44} (2011), no. 22, 225305, 12 pp.


\bibitem{sab} G. Sabidussi, Graph multiplication, \textit{Math. Z.} \textbf{72}  (1960), 446--457. 



\bibitem{stefano} S. Spessato, Pullback functors for reduced and unreduced $L^{q,p}$-cohomology,
\textit{Ann. Global Anal. Geom.} \textbf{62} (2022), 533--578. 







\bibitem{wang} K. Wang, Kronecker products of permutations and its applications to group matrices, \textit{Linear Mult. Algebra} \textbf{9}  (1980), 69–-76.






\end{thebibliography}
\end{document}